\documentclass[a4paper,11pt,reqno]{amsart}

\usepackage[backend=bibtex, giveninits, style=alphabetic, uniquelist=false, doi=false, url=false, isbn = false]{biblatex}
\usepackage{abbrev}
\usepackage{packages}
\usepackage{format}

\title[Nonparametric adaptive Poisson regression]{Nonparametric Poisson regression\\from independent and weakly dependent observations by model selection}

\author{Martin Kroll}

\email{martin.kroll@math.uni-mannheim.de}
\thanks{This paper is an extension of results from the author's PhD thesis~\cite{kroll2017concentration}.
	I am grateful to my supervisors Jan Johannes and Martin Schlather for constructive discussions and support.
	Financial support by the Deutsche Forschungsgemeinschaft (DFG) through the Research Training Group RTG 1953 is gratefully acknowledged.\\
	Present address: \'Ecole Nationale de la Statistique et de l'Administration \'Economique (ENSAE), 5 avenue Henry Le Chatelier, F-91120 Palaiseau. \texttt{martin.kroll@ensae.fr}.}
\address{Universit\"at Mannheim}
\date{\today}

\allowdisplaybreaks % allows page break inside of align environments

\begin{document}
	
\begin{abstract}
	We consider the non-parametric Poisson regression problem where the integer valued response $Y$ is the realization of a Poisson random variable with parameter $\lambda(X)$.
	The aim is to estimate the functional parameter $\lambda$ from independent or weakly dependent observations $(X_1,Y_1),\ldots,(X_n,Y_n)$ in a random design framework.
	
	First we determine upper risk bounds for projection estimators on finite dimensional subspaces under mild conditions.
	In the case of Sobolev ellipsoids the obtained rates of convergence turn out to be optimal.

	The main part of the paper is devoted to the construction of adaptive projection estimators of $\lambda$ via model selection.
	We proceed in two steps: first, we assume that an upper bound for $\Vert \lambda \Vert_\infty$ is known. 
	Under this assumption, we construct an adaptive estimator whose dimension parameter is defined as the minimizer of a penalized contrast criterion.
	Second, we replace the known upper bound on $\Vert \lambda \Vert_\infty$ by an appropriate plug-in estimator of $\Vert \lambda \Vert_\infty$.
	The resulting adaptive estimator is shown to attain the minimax optimal rate up to an additional logarithmic factor both in the independent and the weakly dependent setup.
	Appropriate concentration inequalities for Poisson point processes turn out to be an important ingredient of the proofs.
	
	We illustrate our theoretical findings by a short simulation study and conclude by indicating directions of future research.
\end{abstract}	
	
\maketitle

\section{Introduction}

We consider the non-parametric estimation of a regression function $\lambda: \X \to [0,\infty)$ defined on some Polish space $\X$ from observations $(X_1,Y_1),\ldots,(X_n,Y_n)$
where, conditional on $X_1,\ldots,X_n$, the $Y_i$ are independent and Poisson distributed with parameter $\lambda(X_i)$.
The covariates $X_1,\ldots,X_n$ are drawn from some strictly stationary process $(X_i)_{i\in \Z}$, and we will consider the two cases where either (i)~the $X_1,\ldots,X_n$ are independent, or (ii)~some adequate condition on the dependence of the underlying process $(X_i)_{i\in \Z}$ is satisfied.
Although we will also provide minimax theoretical results for the non-parametric estimation problem at hand, our focus will be on the adaptive estimation of $\lambda$, that is, the construction of estimators that depend only on the observations but not on any structural presumptions concerning the regression function.

Regression models with count data, i.e., non-negative and integer-valued response, are of interest in a wide range of applications, for instance in economics~\cite{winkelmann2008econometric}, quantitative criminology~\cite{berk2008overdispersion}, and ecology~\cite{ver_hoef2007quasi-poisson}.
The \emph{Poisson regression} model introduced above is the most natural example of such a count data regression model.
Other models with count data response include models based on the negative binomial distribution which can also deal with \emph{overdispersion}.
Such more advanced models will not be considered in this paper.
Most of the work in the area of count data regression has been devoted to parametric models, see for instance the monograph~\cite{cameron1998regression} for a comprehensive overview of methods.
Let us just mention some examples:
the paper~\cite{diggle1998model} gives an application of a Poisson regression model in a geostatistical context.
It provides a fully parametric approach and suggests MCMC techniques for fitting a model to the given data.
The paper~\cite{carota2002semiparametric} introduces a semi-parametric Bayesian model for count data regression and applies it as a prognostic model for early breast cancer data.
The article~\cite{nakaya2005geographically} considers geographically weighted Poisson regression for disease association mapping.

Despite its potential utility in many applications, non-parametric Poisson regression has hardly been studied from a theoretical point so far.
One possible approach is to apply the so-called Anscombe transform~\cite{anscombe1948transformation} to the data and treat the data as if they were Gaussian.
Another approach would be to consider the generalized linear model representation of Poisson regression and allow for varying coefficients~\cite{hastie1993varying-coefficient,fan1999statistical}.
Recent work has also considered the Poisson regression model in a high-dimensional framework using the LASSO and the group LASSO~\cite{ivanoff2016adaptive}.
Another interesting reference is~\cite{fryzlewicz2008datadriven}: in this paper the author considers a very general model with Poisson regression as a special case.
In contrast to our model, only regression with deterministic design is considered (note that the distinction between independent and weakly dependent covariates considered by us is not possible in the model with deterministic design).
Moreover, the automatic choice of the smoothing parameter is not addressed from a theoretical point of view in~\cite{fryzlewicz2008datadriven} whereas this is the major topic of our contribution.
Finally, let us mention the work~\cite{besbeas2004comparative} where an extensive simulation study for count data regression using wavelet methods was performed.
That paper contains also further references to Bayesian methods in the context of count data regression.
In this paper, we study adaptive non-parametric Poisson regression via the model selection approach.
To the best of our knowledge, this approach has not been used for non-parametric Poisson regression so far
%Besides the paper~\cite{ivanoff2016adaptive}, there does not seem to exist another contribution that considers non-parametric Poisson regression via model selection.
(in a parametric framework, however, the recent paper~\cite{kamo2013bias-corrected} considers a model selection approach via a bias-corrected AIC criterion).

Note that a characteristic feature of the non-parametric Poisson regression model is the fact that it naturally incorporates heteroscedastic noise.
Besides work on regression in presence of homoscedastic errors (see for instance~\cite{baraud2000model}), there already exists research that considers model selection techniques in regression frameworks containing heteroscedasticity~\cite{saumard2013optimal}.
However, in~\cite{saumard2013optimal} the observations are of the form
\begin{equation*}
	Y = r(X) + \sigma(X) \epsilon
\end{equation*}
where $r$ is the unknown regression function to be estimated, the residuals $\epsilon$ have zero mean and variance one, and the function $\sigma$ models the unknown heteroscedastic noise level.
Note that this model does not contain the Poisson regression model to be considered in this paper as a special case.

Our paper is also more general with regard to another aspect: we do not exclusively stick to the case that the covariates $X_i$ are independent but also consider the more general case that the covariates are weakly dependent which seems to be more realistic at least in some real world scenarios.
For instance, when studying clutch sizes of bird eggs that are modeled via count data models (that usually go beyond the Poisson model studied here due to over-dispersion of the data) in ornithology~\cite{ridout2004empirical}, the covariates (e.g., temperature) are not independent when data are collected over a period of time. 
Concerning mathematical methodology, we will model this by imposing throughout conditions on the decay of the so-called $\beta$-mixing coefficients.
The class of time series with $\beta$-mixing coefficients is sufficiently large to be of interest for applications and includes stationary vector ARMA processes~\cite{mokkadem1990proprietes} or even more general autoregressive processes of the form $X_t = m(X_{t-1}) + \sigma(X_{t-1})\epsilon_t$ under mild conditions on the functions $m$ and $\sigma$ \cite{neumann2006asymptotic,doukhan1994mixing}.
Our methodological approach is mainly based on fundamental results from the article~\cite{viennet1997inequalities} that have also been exploited in a wide variety of other statistical problems:
in~\cite{baraud2001adaptive} and~\cite{asin2016adaptive_a} the authors consider the non-parametric estimation of a regression function in case of $\beta$-mixing covariates.
The paper~\cite{lacour2008adaptive} considers adaptive estimation of the transition density of a particular hidden Markov chain under the assumption that the hidden chain is $\beta$-mixing.
From a methodological point of view our approach was also inspired by the recent work~\cite{asin2016adaptive_a}.
However, in contrast to that paper, we build our construction of adaptive estimators on the model selection technique from~\cite{barron1999risk} only, whereas \cite{asin2016adaptive_a} combines the model selection approach with a more recent technique due to Goldenshluger and Lepski~\cite{goldenshluger2011bandwidth}.

Let us sketch the organisation and summarize the main contributions of the paper.
In Section~\ref{sec:methodology} we introduce notations and the general methodology used in the paper.
Section~\ref{sec:independent} is devoted to the case of independent observations: we derive a general minimax upper bound and a matching lower bound over Sobolev ellipsoids.
We then consider adaptive estimation of the regression function via model selection which has not been addressed before in the literature.
We first consider an estimator based on the a priori knowledge of an upper bound on the regression function (Subsection~\ref{subsubsec:known:xi}), and then, inspired by the approach in~\cite{comte2001adaptive}, put some effort to develop an estimator that does not depend on this assumption (Subsection~\ref{subsubsec:unknown:xi}).
The risk bound of the adaptive estimators is deteriorated by a logarithmic factor only in comparison with the minimax optimal rate.
In Section~\ref{sec:dependent} we extend the findings from Section~\ref{sec:independent} to the weakly dependent case. The proofs in this case are more demanding than in the independent case, but the results are essentially the same.
Subsequent to our theoretical findings, we provide a short simulation study in Section~\ref{sec:simulations}.
In Section~\ref{sec:discussion} we conclude and discuss perspectives for future research.
All proofs are deferred to the appendix.

\section{Methodology}\label{sec:methodology}

\subsection{Notation} 
Throughout the paper, let $(\X, \Xs, \P)$ be a fixed probability space and denote by $L^2=L^2(\X, \Xs, \P)$ the space of square-integrable random variables.
The space $\X$ is assumed to be Polish with $\Xs$ being the $\sigma$-field generated by the topology of $\X$.
The regression function $\lambda$ is always assumed to belong to $L^2$.
For $p=1,2$, let $\Vert \cdot \Vert_p$ denote the usual $L^p$ norm, i.e., $\Vert g \Vert_p = \left( \int_\X \vert g \vert^p d\P \right)^{1/p}$ (in the case $p=2$ we usually suppress the index $p$).
In the special case $p=2$, we denote the scalar product corresponding to the norm $\Vert \cdot \Vert_2$ by $\langle \cdot, \cdot \rangle$.
$\Vert \cdot \Vert_\infty$ denotes the sup norm on the space $\X$.
We write $a_n \lesssim b_n$ if $a_n \leq Cb_n$ holds for all $n \in \N$ with some constant independent of $n$.

\subsection{Projection estimators}
Let $\Sc_n$ be a finite-dimensional subspace of $L^2$.
For a subspace $\Sc_\mf \subseteq \Sc_n$ with orthonormal basis $\{  \phi_\eta \}_{\eta \in \Ic_\mf}$ ($\Ic_\mf$ being an appropriate index set of cardinality $\D_\mf$ equal to the dimension of the model) we denote by $\widehat \lambda_\mf$ the projection estimator given through
\begin{equation}\label{eq:def:proj:est}
	\lambdahat_\mf = \sum_{\eta \in \Ic_ \mf} \thetahat_\eta \phi_\eta
\end{equation}
where $\thetahat_\eta = \frac{1}{n} \sum_{i=1}^{n} Y_i \phi_\eta(X_i)$ is an unbiased estimator of the true generalized Fourier coefficient $\theta_\eta = \int \lambda(x) \phi_\eta(x) \P (dx)$.
As often in non-parametric statistics, the theoretical investigation of the estimator $\lambdahat_\mf$ will be based on the bias-variance decomposition of the mean integrated squared error
\begin{equation*}
	\E \Vert \lambdahat_\mf - \lambda \Vert^2 = \Vert \lambda - \lambda_\mf \Vert^2 + \E \Vert \lambdahat_\mf - \lambda_\mf \Vert^2 =  \Vert \lambda - \lambda_\mf \Vert^2 + \sum_{\eta \in \Ic_\mf} \E [ (\thetahat_\eta - \theta_\eta)^2 ]
\end{equation*}
where $\lambda_\mf = \sum_{\eta \in \Ic_\mf} \theta_\eta \phi_\eta$ denotes the orthogonal projection of $\lambda$ on the space $\Sc_\mf$. 
For our theoretical treatment we impose the following condition on the models $\Sc_\mf = \spn \{\phi_\eta\}_{\eta \in \Ic_\mf}$.

\begin{assumption}\label{ass:model}
	There exists a positive constant $\Phi$ such that for any $f \in \Sc_\mf$ it holds $\Vert f \Vert_\infty \leq \Phi \sqrt{\D_m} \Vert f \Vert$.
\end{assumption}

\begin{remark}
	As remarked by~\cite{viennet1997inequalities}, Assumption~\ref{ass:model} is equivalent to the assumption that for \emph{any} orthonormal basis $\{ \phi_\eta \}_{\eta \in \Ic_\mf}$ of $\Sc_\mf$ it holds $\Vert \sum_{\eta \in \Ic_\mf} \phi_\eta^2 \Vert_\infty \leq \Phi^2 \D_\mf$.
	In our proofs, we will exploit this characterization of Assumption~\ref{ass:model}.
	The class of models satisfying Assumption~\ref{ass:model} incorporates, for instance, all bounded bases (such as the trigonometrical basis) as well as piecewise polynomials, splines, and wavelets (see~\cite{viennet1997inequalities}, p.~475 for further details).
\end{remark}

When studying adaptive estimators we have to impose further conditions on the set of potential models (see Assumption~\ref{ass:model:adaptive}).

\subsection{Dependency assumptions}

In this paper, we aim at developing the theory both for independent and weakly dependent observations.
In order to describe dependency between subsequent observations of the covariates $X_1,\ldots,X_n$, several concepts of mixing coefficients have been introduced (see~\cite{bosq1998nonparametric} for a comprehensive introduction):
%In order to formulate the dependence structure between random variables, several concepts of mixing coefficients have been introduced.
in this paper we consider the $\beta$-mixing (or \emph{absolutely regular-mixing}) coefficients that were originally introduced in~\cite{kolmogorov1960strong}.
For a probability space $(\Omega, \As, \Q)$ and two sub-$\sigma$-fields $\Us$ and $\Vs$ of $\As$ the $\beta$-mixing coefficient is defined by
\begin{equation*}
	\beta(\Us, \Vs) = \frac{1}{2} \sup \left\lbrace  \sum_i \sum_j \vert \Q(U_i) \Q(V_j) -\Q (U_i \cap V_j) \vert \right\rbrace 
\end{equation*}
where the supremum is taken over all finite partitions $(U_i)_{i \in I}$ and $(V_j)_{j \in J}$ of $\Omega$ which are measurable with respect to $\Us$ and $\Vs$, respectively.
For random variables $X_1, X_2$ we define $\beta(X_1, X_2)$ as the $\beta$-mixing coefficient between the $\sigma$-fields generated by $X_1$ and $X_2$, respectively, i.e., $\beta(X_1, X_2) = \beta(\sigma(X_1), \sigma(X_2))$.
For a strictly stationary process $(X_i)_{i \in \Z}$ of random elements in a Borel space, let us denote $\Fs_0=\sigma(\{X_i\}_{i \leq 0})$ and $\Fs_k= \sigma(\{X_i\}_{i \geq k})$ for any positive integer $k$.
The sequence of $\beta$-mixing coefficients $(\beta_k)_{k \geq 0}$ of the process $(X_i)_{i \in \Z}$ is defined by $\beta_k = \beta(\Fs_0, \Fs_k)$ and the sequence $(X_i)_{i \in \Z}$ is called $\beta$-mixing (or \emph{absolutely regular}) if $\beta_k \to 0$ as $k \to \infty$.

Examples of $\beta$-mixing processes are given in~\cite{lacour2008adaptive} and include, for instance, autoregressive processes of order $1$.
Concerning the minimax theory, the main difficulty in the weakly dependent in contrast to the independent case is to find suitable bounds for the variance of the estimated coefficients $\thetahat_\eta$ in the bias-variance decomposition.
The main tool to deal with this problem will be Lemma~\ref{lem:var:bound:beta} below.
%For the theoretical analysis of the adaptive estimator in the weakly dependent case (see Theorem~\ref{thm:adaptive:dep} and its proof) we have to rely on further findings from~\cite{viennet1997inequalities}.

\subsection{Adaptive estimation via model selection}

Given a finite collection of models $\Mc_n$, the model corresponding to the optimal estimator from the set $\{\lambdahat_\mf\}_{\mf \in \Mc_n}$ depends on the unknown $\lambda$ and is thus not accessible.
Often it is possible to choose an optimal model when imposing smoothness restrictions on the unknown $\lambda$ (see, for instance, Example~\ref{example:sobolev} below) but even this assumption usually seems to be too hard in practise where one wants to construct an estimator of $\lambda$ in a fully-data driven way.
In applications, cross-validation techniques~\cite{arlot2010survey} are quite popular.
There are several other methods that aim at the construction of one single estimator from a given set of estimators, among them Lepski's method~\cite{lepski1991problem}, aggregation~\cite{bunea2007aggregation,lecue2009aggregation,rigollet2012sparse} or model selection~\cite{barron1999risk}.
In this paper, we exclusively stick to the non-parametric model selection approach that was mainly developed in the 1990s (see~\cite{barron1999risk,birge1997model,massart2007concentration} for comprehensive representations of the subject).

The principal idea of the model selection approach is to choose a model $\mfhat$ from the collection $\Mc_n$ by means of a so-called penalized contrast criterion
\begin{equation*}
	\mfhat = \argmin_{\mf \in \Mc_n} \{ \Upsilon(\lambdahat_\mf) + \pen(\mf) \}
\end{equation*}
where $\Upsilon$ is the contrast function and $\pen$ the penalty (in case of non-uniqueness of the minimizer, one chooses an arbitrary one).
Usually, one can prove for the adaptive estimator $\lambdahat_\mfhat$ oracle inequalities of the form
\begin{equation}\label{eq:oracle:informal}
	\E \Vert \lambdahat_\mfhat - \lambda \Vert^2 \lesssim \inf_{\mf \in \Mc_n} \{ \Vert \lambda - \lambda_\mf \Vert^2 + \pen(\mf) \} + \text{'terms of lower order'}. 
\end{equation}
% The 'terms of lower order' often attain the order of the parametric rate, i.e. $n^{-1}$ in a framework with an $n$-sample of observations.
The general form of the result already shows that if one is able to choose the penalty term of the same order as the variance under the model $\mf$, the term over which the infimum is taken mimicks the bias-variance trade-off and one obtains an estimator that attains the optimal rate of convergence.
In our case, we will have to introduce an additional logarithmic factor in the penalty leading to a deterioration of the optimal rate by this logarithmic factor in the adaptive case.
A crucial tool in order to prove oracle inequalities of the above form are suitable concentration inequalities, and in our Poisson regression setup we will use concentration results for Poisson processes.
More precisely, our theoretical analysis is based on a special consequence of a Talagrand type inequality (our Lemma~\ref{lem:conc:cox}) that has turned out to be fruitful in non-parametric estimation (see for instance~\cite{chagny2013estimation}, Proposition~2.2) but has only been transferred to the Poisson setup in~\cite{kroll2017concentration}.

Concerning the case of dependent covariates, an additional difficulty appears by the fact that concentration inequalities which are crucial for the proof in the independent case are not available in the dependent one.
In order to deal with this case, we will exploit a construction due to~\cite{viennet1997inequalities} where the sample $X_1,\ldots,X_n$ is substituted by another sample $X_1^\ast,\ldots,X_1^\ast$ such that non-neighbouring blocks of a certain size of the $X_i^\ast$ are independent.
Simultaneously, the $X_i^\ast$ are constructed such that they coincide with the original $X_i$ with high probability. 
Concentration inequalities will then be applied to the independent blocks instead of to the original $X_i$.
The adaptive estimation will be discussed in detail in  Subsections~\ref{subs:ind:adaptive} (independent case) and~\ref{subs:dep:adaptive} (weakly dependent case), respectively.

\section{Independent observations}\label{sec:independent}

We first focus on the case of independent observations.
In this case we will, besides an upper bound on the risk, also derive a lower bound which shows the optimality of the estimator over certain classes of $L^2$-ellipsoids.
Subsection~\ref{subs:ind:adaptive} is then dedicated to the adaptive estimation by model selection.

\subsection{Upper bound}

The following proposition states an upper bound for a general model.

\begin{proposition}\label{prop:ind:upper}
	For a model $\mf$ satisfying Assumption~\ref{ass:model}, let $\lambdahat_\mf$ be the corresponding projection estimator defined in~\eqref{eq:def:proj:est}.
	Then
	\begin{equation*}
		\E \Vert \lambdahat_\mf - \lambda \Vert^2 \leq \Vert \lambda_\mf - \lambda \Vert^2 + \frac{\Phi^2 \D_\mf}{n} \cdot (\Vert \lambda \Vert^2 + \Vert \lambda \Vert_1).
	\end{equation*} 
\end{proposition}

Unfortunately, given a collection of models, a model optimizing the upper bound cannot be specified in advance since the bias term depends on the unknown regression function.
The aim of the following example is to illustrate the general result of Proposition~\ref{prop:ind:upper} in the special case when the models under consideration are given by nested spaces generated by the trigonometric basis.
As usual in non-parametric statistics, one can determine a rate optimal model by imposing some \emph{a priori} smoothness conditions on the regression function.

\begin{example}[$L^2$-ellipsoids]\label{example:sobolev}
In order to illustrate the result of Proposition~\ref{prop:ind:upper}, let us consider the special case of the trigonometric basis on the space $(\X, \Xs, \P)=([0,1], \Bc([0,1]), dx)$:
for $\mf \in \N_0$, let $\Ic_\mf=\{ -m,...,m \}$ and $\Sc_m = \spn \{ \phi_\eta \}_{\eta \in \Ic_\mf}$ where
\begin{equation*}
	\phi_0 \equiv 1, \qquad \phi_j(x) = \sqrt{2} \cos(2\pi j x), \qquad \text{and} \qquad \phi_{-j}(x)= \sqrt{2} \sin(2\pi j x) \qquad \text{for}\qquad j \in \N.
\end{equation*}
Smoothness of the regression function $\lambda$ may be expressed by assuming its membership to a suitable ellipsoid
\begin{equation*}
	\Theta_\gamma^R = \bigg\{ \lambda = \sum_{j \in \Z} \theta_j \phi_j \in L^2 : \lambda \geq 0 \text{ and } \sum_{j \in \Z} \theta_j^2 \gamma_j^2 \leq R \bigg\}
\end{equation*}
where $R > 0$ and $\gamma = (\gamma_j)_{j \in \Z}$ is a strictly positive symmetric sequence such that $\gamma_0=1$ and the sequence $(\gamma_n)_{n \in \N_0}$ is non-decreasing.
Typical examples of $\gamma$ are $\gamma_j=\vert j \vert^p$ (for $j \geq 1$) and $\gamma_j = \exp(p \vert j \vert)$ for $p \geq 0$.
Under the stated assumption on the sequence $\gamma$, the bias term in the proof of Proposition~\ref{prop:ind:upper} may be bounded as follows (as introduced above, $\lambda_\mf$ denotes the projection of $\lambda$ on $\Sc_\mf$):
\begin{equation*}
	\Vert \lambda_\mf - \lambda \Vert^2 = \sum_{\vert j \vert > \mf} \theta_j^2 \leq \gamma_\mf^{-2} \sum_{\vert j \vert > \mf} \theta_j^2 \gamma_j^2 \leq R \gamma_\mf^{-2}.
\end{equation*}
The trade-off between squared bias and variance is thus equivalent to the best compromise between $\gamma_\mf^{-2}$ and $\mf n^{-1}$. 
In the polynomial case $\gamma_j = \vert j \vert^{p}$, the best compromise is realized by $\mopt \approx n^{1/(2p+1)}$, and we get the classical non-parametric rate $n^{-2p/(2p+1)}$.
In the exponential case $\gamma_j = \exp(p\vert j \vert)$, we have $\mopt \asymp \log n$ and the rate is $\log n/n$.
\end{example}

\subsection{Lower bound for Sobolev ellipsoids}

The following theorem provides a lower bound on the minimax risk in the framework of Example~\ref{example:sobolev}.

\begin{theorem}\label{thm:lower}
	Consider $\Theta_\gamma^R$ defined as in Example~\ref{example:sobolev}.
	Let $\gamma=(\gamma_j)_{j \in \Z}$ be a strictly positive symmetric sequence such that $\gamma_0 = 1$ and the sequence $(\gamma_n)_{n \in \N_0}$ is non-decreasing.
	Set $\mopt = \argmin_{k \in \N_0} \max \{ \gamma_k^{-2}, \frac{2k+1}{n} \}$ and $\Psi_n = \max\{ \gamma_\mopt^{-2}, \frac{2\mopt+1}{n} \}$.
	Assume that
	\begin{enumerate}
		\item\label{PR:it:C1} $\Gamma \defeq \sum_{j \in \Z} \gamma_j^{-2} < \infty$, and
		\item\label{PR:it:C2} $0 < \eta^{-1}\defeq \inf_{n \in \N} \Psi_n^{-1} \min \{ \gamma_{\mopt}^{-2}, \frac{2\mopt+1}{n} \}$
		for some $1 \leq \eta < \infty$.
	\end{enumerate}
	Then, for any $n \in \N$,
	\begin{equation*}
	\inf_{\widetilde \lambda} \sup_{\lambda \in \Theta_\gamma^R} \E  \Vert \widetilde \lambda - \lambda \Vert^2  \gtrsim \Psi_n
	\end{equation*}
	where the infimum is taken over all estimators $\widetilde \lambda$ of $\lambda$ based on the observation of the tuples $(X_1,Y_1),\ldots,(X_n,Y_n)$.
\end{theorem}

The heuristic behind condition~\eqref{PR:it:C2} is that for the optimal model $\mopt$ the corresponding squared bias $\gamma_{\mopt}^{-2}$ and the variance $\frac{2\mopt+1}{n}$ should be of the same order.
It is satisfied for both the case that $\gamma_j = \vert j \vert^{p}$ and the case that $\gamma_j = \exp(p\vert j \vert)$.

The lower bound of Theorem~\ref{thm:lower} shows together with Example~\ref{example:sobolev} that under the given assumptions the rate $\Psi_n$ is optimal.

\subsection{Adaptive estimation}\label{subs:ind:adaptive}

In order to construct an adaptive estimator of the regression function we stick to the model selection method sketched in the introduction.
For this approach we define the contrast function
\begin{equation}\label{eq:def:contrast}
	\Upsilon_n(f) = \Vert f \Vert^2 - 2 \langle \lambdahat_n,f \rangle
\end{equation}
for $f \in L^2$ and $\lambdahat_n$ is the projection estimator associated to the subspace $\Sc_n$.

The aim of the model selection approach is to select from a given collection $\Mc_n$ of submodels of $\Sc_n$ in a completely data-driven way a candidate that behaves as well as possible as the best model in the collection in the sense of an oracle inequality like~\eqref{eq:oracle:informal}.
In order to establish our theoretical results, we have to introduce a further assumption on the collection of models.

\begin{assumption}\label{ass:model:adaptive}
	The models $\mf \in \Mc_n$ are nested in the sense that for $\mf, \mf' \in \Mc_n$ the inequality $\D_\mf \leq \D_{\mf'}$ implies that $\Sc_\mf \subseteq \Sc_{\mf'}$ (in particular, this implies that there is at most one model with $\D_\mf = d$ for a given dimension $d \in \N$).
	In addition, there exist universal constants $c_\mf, c_{\Mc} > 0$ such that
	\begin{itemize}
		\item $\D_\mf \leq c_\mf n$ for all $\mf \in \Mc_n$ and $n \in \N$,
		\item $\vert \Mc_n \vert \leq c_{\Mc} n$ for all $n \in \N$.
	\end{itemize}
	The subspace associated with the maximal model in the collection $\Mc_n$ will be denoted with $\Sc_n$ and the corresponding basis with $\{ \phi_\eta \}_{\eta \in \Ic_n}$.
\end{assumption}

The nestedness assumption together with Assumption~\ref{ass:model} is quite standard (see \cite{birge1997model}, p.~58) and satisfied, e.g., by the trigonometric basis, piecewise polynomials and wavelets (see \cite{birge1997model}, p.~71).

\subsubsection{Known upper bound of the regression function}\label{subsubsec:known:xi}

Before we derive a fully adaptive estimator we first stick to the following assumption.
\begin{assumption}\label{ass:xi:known}
	We have access to some $\xi > 0$ such that $\Vert \lambda \Vert_\infty \leq \xi$.
\end{assumption}
Based on the knowledge of $\xi$, we define for a model $\mf \in \Mc_n$ the penalty:
\begin{equation}\label{eq:pen:xi:known}
	\pen(\mf) = 24 \mu \cdot \frac{\Phi^2 \D_\mf}{n} + 400 \mu \cdot \Phi^2\D_\mf \cdot \frac{\log(n+2)}{n}
\end{equation}
where $\mu  = 1 \vee \xi^2$.
Then a model $\mftilde$ is chosen as follows:
\begin{equation*}
	\mftilde = \argmin_{\mf \in \Mc_n} \{ \Upsilon(\lambdahat_\mf) + \pen(\mf) \}
\end{equation*}
(in case that the minimizer is not unique one chooses an arbitrary one).

\begin{remark}
 Note that the definition of the penalty in~\eqref{eq:pen:xi:known} (and in all subsequent definitions within this work) contains an additional logarithmic term which is in contrast to the standard Gaussian regression setup where such a factor is not necessary (see~\cite{massart2007concentration}, Section~4.3.3).
 This additional term will cause an extra logarithmic factor in the adaptive upper bounds in contrast to the minimax upper bounds.
 Unfortunately, it is not clear to us if this factor is indeed unevitable or just an artefact of our method.
 A heuristic explanation for the necessity of this extra logarithmic factor in our case would be the fact that the concentration inequalities used for our theoretical analysis are not derived in a sub-gaussian setup but in a Poisson setup with sub-gamma tails.
\end{remark}

\begin{theorem}\label{thm:adaptive:xi:known}
	For every $n \in \N$, let $\Mc_n$ be a collection of models such that Assumption~\ref{ass:model} is satisfied for all $\mf \in \Mc_n$. Further assume that the collection $\Mc_n$ satisfies Assumption~\ref{ass:model:adaptive} and that Assumption~\ref{ass:xi:known} holds.
Then
	\begin{equation*}
		\E \Vert  \lambdahat_\mftilde - \lambda \Vert^2 \lesssim \min_{\mf \in \Mc_n} \max \{ \Vert \lambda - \lambda_\mf \Vert^2 , \pen(\mf) \} + \frac{1}{n}.
	\end{equation*}
\end{theorem}

\begin{example}[Continuation of Example~\ref{example:sobolev}]
	Setting $\Mc_n = \{ 0,\ldots,n \}$ and defining the spaces $\Sc_\mf$ for $\mf \in \Mc_n$ exactly as in Example~\ref{example:sobolev}, the penalty function reads
	\begin{equation*}
		\pen(\mf) = 24 \mu \cdot \frac{2\mf +1}{n} + 400 \mu \cdot (2\mf + 1) \cdot \frac{\log(n+2)}{n}.
	\end{equation*}
	Up to the additional logarithmic factor in the second summand, the penalty term behaves exactly as the variance.
	Hence, the adaptive estimate attains the optimal rate up to this extra logarithmic factor.
\end{example}

\begin{remark}
	A careful inspection of the proof of Theorem~\ref{thm:adaptive:xi:known} shows that the established bound holds uniformly over sets $\Theta_\gamma^R \cap \{ \Vert \lambda \Vert_\infty \leq \xi \}$ for $\xi > 0$.
	In the framework of Example~\ref{example:sobolev} the rate is then deteriorated by an additional logarithmic factor in the case of polynomially increasing $\gamma$ and optimal in the case of exponentially increasing $\gamma$.
\end{remark}

\subsubsection{Unknown upper bound of the regression function}\label{subsubsec:unknown:xi}

We now propose an adaptive estimator of the regression function $\lambda$ that does not depend on \emph{a priori} knowledge of an upper bound for $\Vert \lambda \Vert_\infty$, and is thus fully data-driven.
Of course, the key idea is to replace the quantity $\xi$ from Assumption~\ref{ass:xi:known} appearing in the definition of the penalty~\eqref{eq:pen:xi:known} by an appropriate estimator of $\Vert \lambda \Vert_\infty$.
For the construction of the estimator of $\Vert \lambda \Vert_\infty$, we take inspiration from an approach that was used in \cite{comte2001adaptive} in the context of adaptive estimation of the spectral density from a stationary Gaussian sequence.
More precisely, the estimator of $\Vert \lambda \Vert_\infty$ is obtained as the plug-in estimator $\Vert \lambdahat_\Pi \Vert_\infty$ where $\lambdahat_\Pi$ is a suitable histogram estimator of $\lambda$ based on some partition $\Pi = \{ \X_1,\ldots,\X_M \}$ of the space $\X$ in mutually disjoint measurable sets $\X_j$, $j=1,\ldots,M$ with $\X = \bigcup_{j=1}^M \X_j$.
More precisely, $\lambdahat_\Pi$ is defined as $\lambdahat_\Pi =\sum_{j=1}^{M} \pi_j \1_{\X_j}$ where $\pi_j = \frac{1}{n\sqrt{\P (\X_j)}} \sum_{i=1}^n Y_i\1_{\{X_i \in \X_j\}}$.
Obviously, $\lambdahat_\Pi$ is the projection estimator on the space $\Sc_\Pi$ generated by the orthonormal basis functions $\frac{1}{\sqrt{\P (X_j)}}\1_{\X_j}$, $j=1,\ldots,m$ which has dimension $\D_\Pi=\vert \Pi \vert = M$.

We substitute the quantity $\xi$ in the definition of the penalty term defined in~\eqref{eq:pen:xi:known} with $\Vert \lambdahat_\Pi \Vert_\infty$.
More precise assumptions on the partition $\Pi$ will be stated in Theorem~\ref{thm:adaptive:xi:unknown} below.
Further, by adapting the numerical constants in the definition of the penalty (which is necessary for our proof), we replace the deterministic penalty term used under the validity of Assumption~\ref{ass:xi:known} by the random penalty
\begin{equation*}
	\penhat(\mf) = 384  \muhat \cdot \frac{\Phi^2\D_\mf}{n} + 6400 \muhat \cdot \Phi^2\D_\mf \cdot \frac{\log(n+2)}{n}
\end{equation*}
where $\muhat = 1 \vee \Vert \lambdahat_\Pi \Vert_\infty^2$.
Keeping the contrast function $\Upsilon_n$ as defined in~\eqref{eq:def:contrast} we finally put
\begin{equation*}
	\mfhat = \argmin_{\mf \in \Mc_n} \{\Upsilon_n(\lambdahat_\mf) + \penhat(\mf)\}.
\end{equation*}

\begin{theorem}\label{thm:adaptive:xi:unknown}
	For every $n \in \N$, let $\Mc_n$ be a collection of models such that Assumption~\ref{ass:model} is satisfied for all $\mf \in \Mc_n$.
	Further assume that the collection of models satisfies Assumption~\ref{ass:model:adaptive} and that the following conditions hold:
	\begin{enumerate}[label=($\Pi$\arabic*), leftmargin=1.5cm, itemsep=0em]
		\item\label{ass:Pi1} $\Vert \lambda - \lambda_\Pi \Vert_\infty \leq \frac{1}{4} \Vert \lambda \Vert_\infty$ where $\lambda_\Pi$ denotes the projection of $\lambda$ on $\Sc_\Pi$, and
		\item\label{ass:Pi2} the partition $\Pi = \{ \X_1,\ldots,\X_M \}$ in the definition of the auxiliary estimator $\lambdahat_\Pi$ satisfies $\P (\X_j) \geq c_\Pi/M$ for some constant $c_\Pi > 0$ and
		\begin{equation*}
			M \leq \frac{c_\Pi n}{320\log n}.
		\end{equation*}
	\end{enumerate}
	Then
	\begin{equation*}
		\E \Vert \widehat \lambda_\mfhat - \lambda \Vert^2 \lesssim \min_{\mf \in \Mc_n} \max \{ \Vert \lambda - \lambda_\mf \Vert^2, \pen(\mf) \} + \frac{1}{n}
	\end{equation*}
	where $\pen(\mf) = 24 \mu \cdot \frac{\Phi^2\D_\mf}{n} + 400 \mu \cdot \Phi^2\D_\mf \cdot \frac{\log(n+2)}{n}$ and $\mu = {1 \vee \Vert \lambda \Vert_\infty^2}$.
\end{theorem}

The additional Assumptions~\ref{ass:Pi1} and \ref{ass:Pi2} are inspired by similar assumptions made in Theorem~2 of~\cite{comte2001adaptive}.
Note that Assumption~\ref{ass:Pi2} is especially satisfied in the case that $\X$ is some compact subset of some $\R^d$ and $\P$ admits a density with respect to the Lebesgue measure that is bounded from below by some strictly positive constant.
In this case, one can take an arbitrary partition of $\X$ into $M$ sets $\X_1,\ldots,\X_M$ of equal Lebesgue measure.

\section{Dependent observations}\label{sec:dependent}

\subsection{Upper bound}

The following proposition provides an upper bound on the risk in the weakly dependent case.

\begin{proposition}\label{prop:dep:upper}
	For a model~$\mf$ satisfying Assumption~\ref{ass:model} let $\lambdahat_\mf$ be the corresponding projection estimator. Then
	\begin{equation*}
		\E \Vert \lambdahat_\mf - \lambda \Vert^2 \leq \Vert \lambda_\mf - \lambda \Vert^2 + \frac{\Phi^2 \D_\mf}{n} \left[ \Vert \lambda \Vert_1 + 4 \Vert \lambda \Vert_\infty^2 \left( \sum_{k=0}^{n} \beta_k \right) \right].
	\end{equation*} 
\end{proposition}

Note that the numerical constant is uniformly bounded in $n$ under the assumption that $\sum_{k=0}^{\infty} \beta_k < \infty$.

\begin{remark}
	Under the assumption that $\sum_{k=0}^{\infty} \beta_k < \infty$ the bound in the weakly dependent case coincides (apart from the numerical constants appearing) with the one in the independent case.
	Moreover, since the weakly dependent case incorporates the independent one, the lower bound given in Theorem~\ref{thm:lower} provides also the benchmark for the weakly dependent case.
\end{remark}

\subsection{Adaptive estimation}\label{subs:dep:adaptive}

We now consider adaptive estimation in case of weak dependence.
We stick directly to the case that Assumption~\ref{ass:xi:known} is not satisfied.
We keep the contrast function from Section~\ref{sec:independent} but define the penalty as
\begin{equation*}
	\penhat(\mf) = \D_\mf \cdot \frac{\log(n+2)}{n} + 6400 \muhat \cdot \Phi^2\D_\mf \cdot \frac{\log(n+2)}{n}.
\end{equation*}

In addition we have to impose the following assumption on the $\beta$-mixing coefficients which is similar to Assumption~A4 in~\cite{lacour2008adaptive}.

\begin{assumption}\label{ass:ex:lacour} % entspricht Assumption A4 in Lacour 2008, Seite 789
	The process $(X_i)_{i  \in \Z}$ is geometrically $\beta$-mixing ($\beta_q \lesssim e^{-\theta q}$) or arithmetically $\beta$-mixing ($\beta_q \lesssim Mq^{-\theta}$) with $\theta \geq 9$ in the latter case.
\end{assumption}

Note that under Assumption~\ref{ass:ex:lacour} the condition $\sum_{k=0}^{\infty} \beta_k < \infty$ is satisfied.
Examples of processes satisfying Assumption~\ref{ass:ex:lacour} are given in~\cite{lacour2008adaptive} and include autoregressive processes of order $1$.
We have the following theorem.

\begin{theorem}\label{thm:adaptive:dep}
	Let the assumptions of Theorem~\ref{thm:adaptive:xi:unknown} be satisfied
	with the condition on $M$ replaced with
	\begin{equation*}
		M \leq \frac{c_\Pi n^{1/3}}{320 \log n}.
	\end{equation*}
	Further assume that Assumption~\ref{ass:ex:lacour} holds.
	Then, for every $n \in \N$,
	\begin{equation*}
		\E \Vert \lambdahat_\mfhat - \lambda \Vert^2 \lesssim \inf_{\mf \in \Mc_n} \max \{ \Vert \lambda - \lambda_\mf \Vert^2 , \pen(\mf) \} + \frac{1}{n}
	\end{equation*}
	where $\pen(\mf) = \D_\mf \cdot \frac{\log(n+2)}{n}$.
\end{theorem}

\section{Numerical results}\label{sec:simulations}

Although the main focus of the present paper is the derivation of theoretical results we provide a short simulation study which is particularly intended to motivate and stimulate further research on the topic of the paper.
More precisely, we test our approach for the test function $\lambda: [0,1] \to (0,\infty)$ where
\begin{equation*}
	\lambda(x) = (5 + 5 \cos(2\pi x)) \cdot \1 _{[0,0.5]}(x) + 10x \cdot \1_{(0.5,1]}(x).
\end{equation*}
Concerning the covariates $X_i$, we assume $X_i \sim \Uc([0,1])$ where $\Uc([0,1])$ denotes the uniform distribution on the interval $[0,1]$.
In a first experiment, we assume the covariates to be independent, in a second one we introduce dependencies by the means of the following model:
$X_1 \sim \Uc([0,1])$, and then in order to generate $X_2,\ldots,X_n$ one recursively sets for $i=2,\ldots,n$
\begin{equation*}
	X_i = \lfloor 0.5X_{i-1} + \epsilon_i\rfloor
\end{equation*}
where the $\epsilon_i$ are i.i.d. $\sim \Nc (0,1)$.

Concerning the models, we consider approximating spaces in terms of piecewise constant functions (histograms).
%More precisely, in the trigonometric basis case we put $\Mc_n = \{ 0,\ldots,n \}$ and for $\mf \in \Mc_n$ we set $\Sc_\mf = \spn \{ \phi_j \}_{0 \leq \jabs \leq \mf}$ where
%\begin{equation*}
%	\phi_0 \equiv 1, \quad \phi_j(x) = \sqrt{2}\cos(2\pi jx), \quad \text{and} \quad \phi_{-j}(x)=\sqrt{2} \sin(2\pi jx) \quad \text{for }j \in \N
%\end{equation*}
%as in Example~\ref{example:sobolev}.
More precisely, we put $\Mc_n = \{ 0,\ldots, \lfloor \log_2 n \rfloor \}$ and for $\mf \in \Mc_n$ define the space $\Sc_\mf$ as the linear span of the functions $\phi_j(x)$, $j=1,\ldots,2^\mf$ where
\begin{equation*}
    \phi_j(x) = \sqrt{2^\mf} \1_{[\frac{j-1}{2^\mf}, \frac{j}{2^\mf})}(x).
\end{equation*}
Note that the assumptions on the models made by us are satisfied:
Assumption~\ref{ass:model} holds true with $\Phi=1$, and
Assumption~\ref{ass:model:adaptive} is satisfied with $c_\mf = c_\Mc = 1$.
For the sake of convenience, we assume that we a priori know that $\Vert \lambda\Vert_\infty \leq 10$ and consider a penalty of the form
$\pen(\mf) = \kappa \cdot \D_\mf \cdot \xi^2 \cdot \frac{\log n}{n}$
with $\xi^2 = 100$ and $\kappa$ a numerical constant.
Concerning the latter, we test various choices.
It is a phenomenon often recognized in non-parametric model selection that the numerical constant in the definition of the penalty term which is convenient to derive theoretical results is by much too large to obtain reasonable results for samples of small size.
This phenomenon is also encountered in our simulation study.
In contrast to this unpleasant behaviour, the overall method is not very demanding from a computational point of view because we essentially minimize the penalized contrast criterion over the set of admissible models which is at most of order $n$ by Assumption~\ref{ass:model:adaptive}.
%%The benchmark we compare with in our simulation study is the so-called oracle estimator $\lambdahat_{\mforacle}$ where the oracle $\mforacle$ is given by
%\begin{equation*}
%	\mforacle = \argmin_{\mf \in \Mc_n} \E \Vert \lambdahat_\mf - \lambda \Vert^2
%\end{equation*}
%(for each of our different setups we determined the oracle empirically by Monte Carlo experiments).
%Note that the oracle is not accessible to the statistician since it depends on the unknown regression function $\lambda_i$.
The results of our simulation experiments are summarized in Tables~\ref{tab:ind_cov} and~\ref{tab:dep_cov} whereas illustrations are given in Figures~\ref{fig:ind_cov} and~\ref{fig:dep_cov}.

\def\arraystretch{1.1}
\setlength{\tabcolsep}{8pt}
\begin{table}[h]
	\footnotesize
	\centering
	\begin{tabular}{ccccccc}
		  &  \multicolumn{2}{c}{$\kappa=0.08$} & \multicolumn{2}{c}{${\kappa=0.09}$} & \multicolumn{2}{c}{$\kappa=0.10$}\\
		 $n$ & error & \textcolor{darkgray}{sd} & error & \textcolor{darkgray}{sd} & error & \textcolor{darkgray}{sd}\\
		 \hline
		 1024 & 1.2171 & \textcolor{darkgray}{0.8069} & \textbf{1.1901} & 
		 \textcolor{darkgray}{0.7209} & 1.2496 & \textcolor{darkgray}{0.6790}\\
		 2048 & 0.7326 &
		 \textcolor{darkgray}{0.3354} & \textbf{0.7097} & \textcolor{darkgray}{0.2510} & 0.7428 &
		 \textcolor{darkgray}{0.2786}
		\vspace{1em}
	\end{tabular}
	\caption{Performance of the histogram estimator in the case of independent covariates for different values of $\kappa$. Mean error and standard deviation were computed over $500$ independent replicates of the experiment. Minimal errors were obtained for $\kappa=0.09$ for both sample sizes considered.}
	\label{tab:ind_cov}
\end{table}
\def\arraystretch{0.1}
\begin{figure}[h]
	\centering
	\begin{subfigure}{.5\textwidth}
		\centering
		\includegraphics[width=\linewidth]{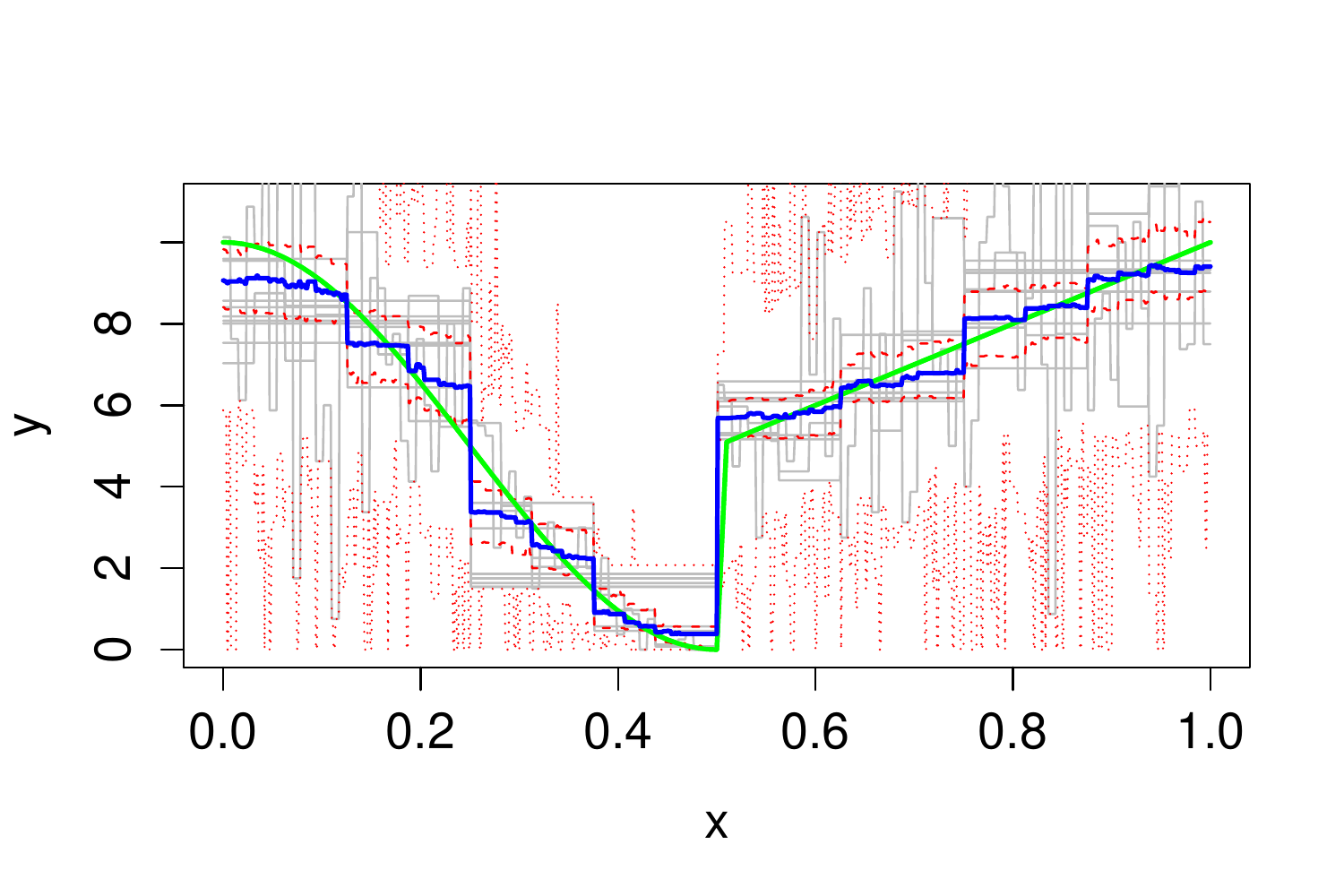}
		\caption{$n=1024, \kappa=0.09$}
		%\label{fig:sub1}
	\end{subfigure}%
	\begin{subfigure}{.5\textwidth}
		\centering
		\includegraphics[width=\linewidth]{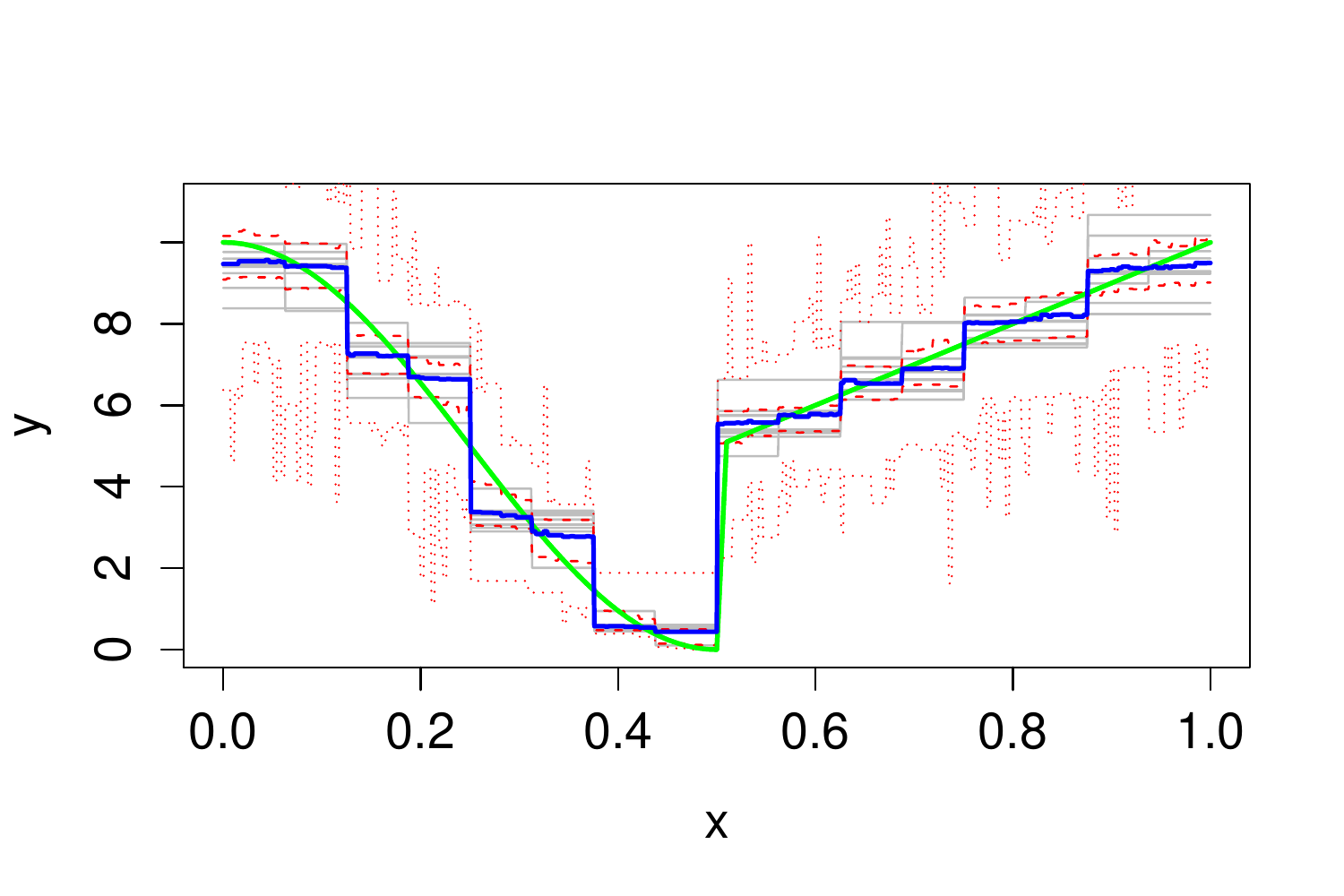}
		\caption{$n=2048, \kappa=0.09$}
		%\label{fig:sub2}
	\end{subfigure}
	\caption{Illustration of the simulation experiment in the case of independent covariates. The true function is plotted in green, whereas the median of the estimator over $n=100$ experiments is in blue. The dashed red lines indicate the empirical pointwise 0.25 resp.\ 0.75 quantile. The dotted red lines indicate the empirical pointwise 0.99 resp.\ 0.01 quantile. The grey lines show some exemplary outcomes of single experiments.}
	\label{fig:ind_cov}
\end{figure}

\def\arraystretch{1.1}
\setlength{\tabcolsep}{9pt}
\begin{table}[h]
	\footnotesize
	\centering
	\begin{tabular}{ccccccccc}
		& \multicolumn{2}{c}{$\kappa=0.08$} & \multicolumn{2}{c}{$\kappa=0.09$} & \multicolumn{2}{c}{$\kappa=0.10$}\\
		$n$ &  error & \textcolor{darkgray}{sd} & error & \textcolor{darkgray}{sd} & error & \textcolor{darkgray}{sd}\\
		\hline
		1024 & 1.2785 & \textcolor{darkgray}{1.1981} & \textbf{1.1505} & \textcolor{darkgray}{0.6012} & 1.2588 & \textcolor{darkgray}{0.6931}\\
		2048 & 0.7112 & \textcolor{darkgray}{0.2899} & \textbf{0.6967} & \textcolor{darkgray}{0.2370} & 0.7375 & \textcolor{darkgray}{0.2733}
		\vspace{1em}
	\end{tabular}
	\caption{Performance of the histogram estimator in the case of dependent covariates for different values of $\kappa$ and different sample sizes $n$. Mean error and standard deviation were computed over $500$ independent replicates of the experiment. Minimal errors were obtained for $\kappa=0.09$ for both sample sizes considered.}
	\label{tab:dep_cov}
\end{table}
\def\arraystretch{1}

\begin{figure}[h]
	\centering
	\begin{subfigure}{.5\textwidth}
		\centering
		\includegraphics[width=\linewidth]{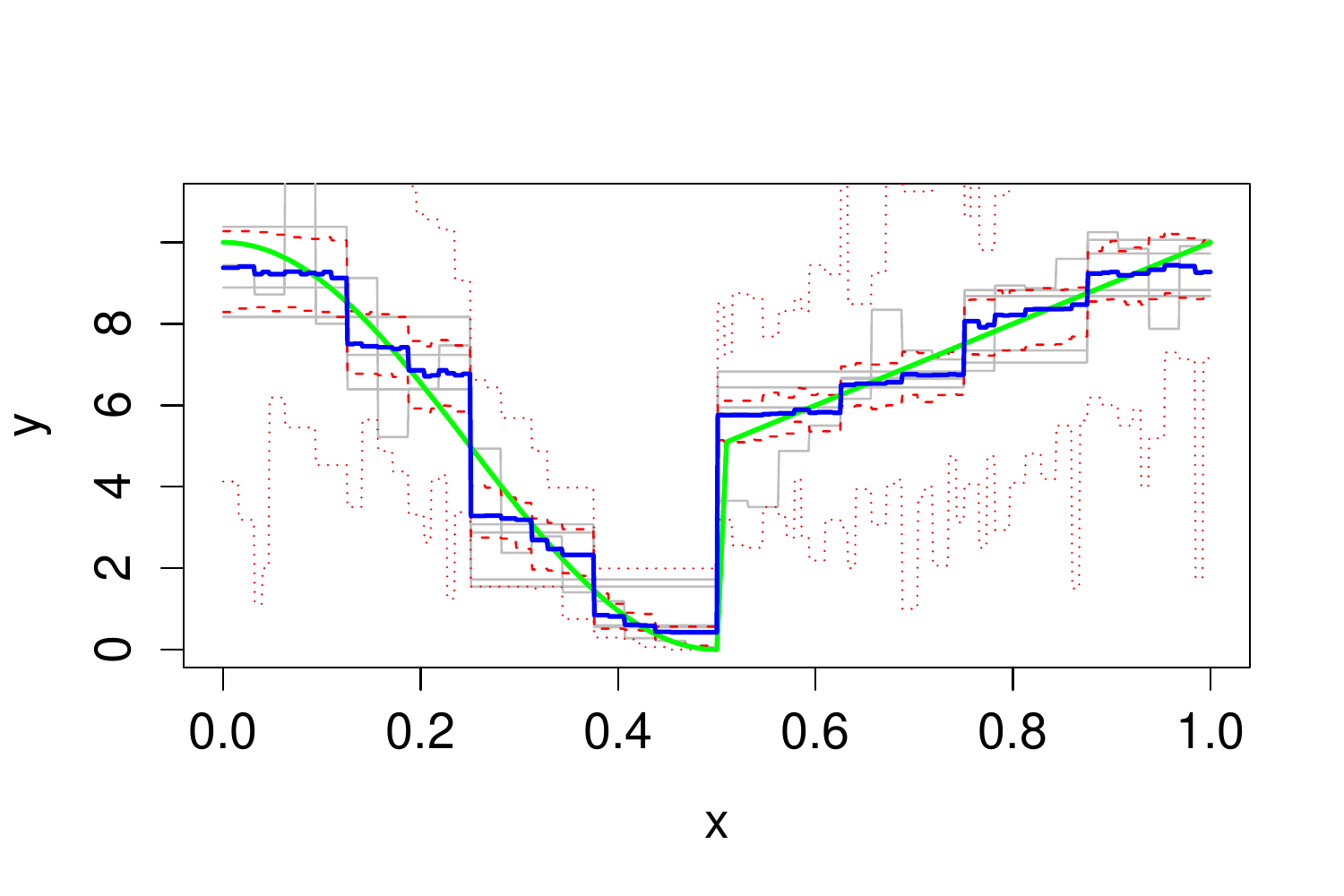}
		\caption{$n=1024, \kappa=0.09$}
		%\label{fig:sub1}
	\end{subfigure}%
	\begin{subfigure}{.5\textwidth}
		\centering
		\includegraphics[width=\linewidth]{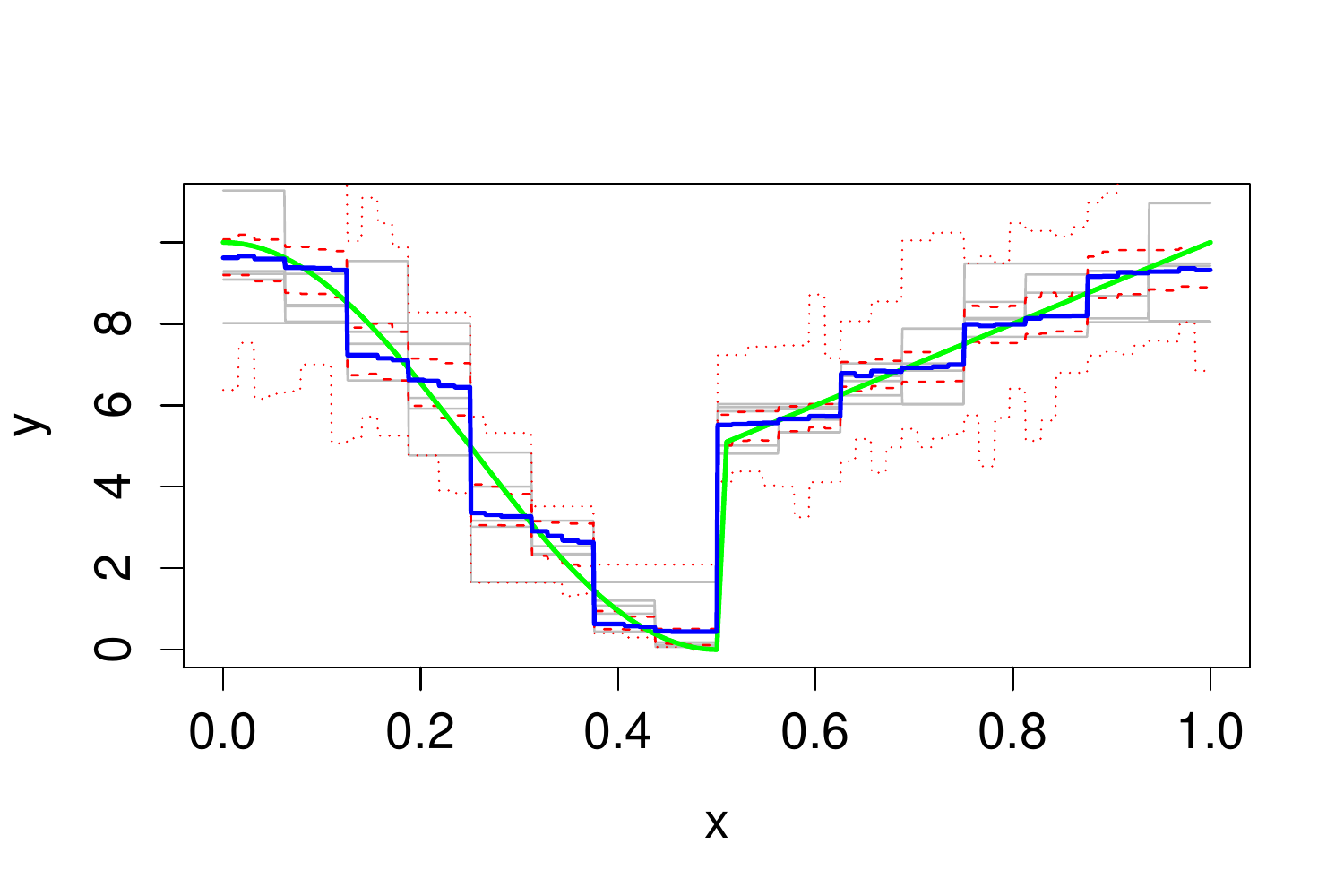}
		\caption{$n=2048, \kappa=0.09$}
		%\label{fig:sub2}
	\end{subfigure}
	\caption{Illustration of the simulation experiment in the case of dependent covariates. The linetypes and colours are chosen as in Figure~\ref{fig:ind_cov}.}
	\label{fig:dep_cov}
\end{figure}

\section{Conclusion and outlook to future research}\label{sec:discussion}

In this paper, we have considered adaptive non-parametric Poisson regression via model selection in the case of independent and $\beta$-mixing covariates.
More precisely, the main objective has been the derivation of oracle inequalities (Theorems~\ref{thm:adaptive:xi:known}, \ref{thm:adaptive:xi:unknown}, and \ref{thm:adaptive:dep}).
A take-home message from our results might be that the price to pay for allowing also weakly dependent covariates is not too high.
Both from a theoretical point of view (slightly harder assumptions in our theorems) as well as in our simulations (there are hardly differences except of fluctuations from one simulation to another) there seems to be no negative impact of weakly depending covariates in contrast to independent ones (note that nearly the same conclusion can be drawn from the results in~\cite{asin2016adaptive_a} whereas \cite{neumann1998strong} provides strong theoretical results in a density estimation context that make our results plausible).
Our simulations have also indicated that the theoretically justified numerical constants appearing in our definition of the penalty terms are by far too large in order to yield good simulation results for moderate sample sizes.
Concerning this aspect, the transfer of very recent results due to Lacour and coauthors (\cite{lacour2016minimal} and \cite{lacour2017estimator}) to our setup might be of interest.
Of course, one could also combine the model selection technique with the recent method by Goldenshluger and Lepski as was done in~\cite{asin2016adaptive_a} for density estimation and Gaussian regression.
However, the simulation study performed in \cite{asin2016adaptive_a} has already shown a comparable performance of this approach to the pure model selection approach considered in the present paper.

Another point of origin for future research might be to consider the case of high-dimensional covariates as was done recently in~\cite{guilloux2016adaptive} and~\cite{guilloux2016adaptive_b}.

The most important open question concerns either the development of an adaptive estimation technique that is capable to obtain the optimal rate of convergence without the additional logarithmic factor or the derivation of a theoretical result showing that this is not possible.
This aspect will be considered in future research as well as the investigation of kernel type estimators instead of projection estimators.

\appendix

\section{Proofs of Section~\ref{sec:independent}}

\subsection{Proof of Proposition~\ref{prop:ind:upper}}

	We have the bias-variance decomposition
	\begin{equation*}
		\E \Vert \lambdahat_\mf - \lambda \Vert^2 = \Vert \lambda_\mf - \lambda \Vert^2 + \E \Vert \lambdahat_\mf - \lambda_\mf \Vert^2.
	\end{equation*}
	Exploiting Assumption~\ref{ass:model} and the independence assumption on the $X_i$ yields for the variance term the estimate
	\begin{align*}
		\E \Vert \widehat \lambda_\mf - \lambda_\mf \Vert^2 &= \sum_{\eta \in \Ic_\mf} \Var(\thetahat_\eta) = \frac{1}{n} \sum_{\eta \in \Ic_\mf} \Var(Y_1\phi_\eta(X_1))\\
		&\leq \frac{1}{n} \E \left[ \sum_{\eta \in \Ic_\mf} Y_1^2\phi_\eta^2(X_1) \right] \leq \frac{\Phi^2\D_\mf}{n} \cdot \E Y_1^2\\
		&\leq \frac{\Phi^2\D_\mf}{n} \cdot (\Vert \lambda \Vert^2 + \Vert \lambda \Vert_1),
	\end{align*}
	and hence the result follows.

\subsection{Proof of Theorem~\ref{thm:lower}}

	For each $\tau = (\tau_j)_{0 \leq \vert j \vert \leq \mopt} \in \{ \pm 1 \}^{2\mopt + 1}$ we define the function $\lambda_\tau$ as
	\begin{align*}
		\lambda_\tau &= \left( \frac{R}{4} \right)^{1/2} + \tau_0 \left( \frac{R\zeta}{16n} \right)^{1/2} + \left( \frac{R\zeta}{16n} \right)^{1/2} \sum_{1 \leq \jabs \leq \mopt} \tau_j \phi_j\\
		&= \left( \frac{R}{4} \right)^{1/2} + \left( \frac{R\zeta}{16n} \right)^{1/2} \sum_{0 \leq \jabs \leq \mopt} \tau_j \phi_j
	\end{align*}
	where $\zeta = \min \{ 1/(\Gamma \eta), 2/R \}$ and the $\phi_j$ are defined as in Example~\ref{example:sobolev}.
	We have
	\begin{align*}
		\left\lVert \left( \frac{R\zeta}{16n} \right)^{1/2} \sum_{0 \leq \jabs \leq \mopt} \tau_j \phi_j \right\rVert_\infty &\leq \left( \frac{R\zeta}{16n} \right)^{1/2} \sum_{0 \leq \jabs \leq \mopt} \sqrt 2\\
		&\leq \left( \frac{R\zeta}{8} \right)^{1/2} \left( \sum_{0 \leq \jabs \leq \mopt} \gamma_j^{-2} \right)^{1/2}  \left( \sum_{0 \leq \jabs \leq \mopt} \frac{\gamma_j^2}{n} \right)^{1/2}\\
		&\leq \left( \frac{\Gamma R \zeta}{8} \right)^{1/2} \left( \gamma_\mopt^2 \cdot \frac{2\mopt + 1}{n} \right)^{1/2}
		\leq \left( \frac{\Gamma R \zeta \eta}{n} \right)^{1/2} \leq \left( \frac{R}{8}\right)^{1/2},
	\end{align*}
	and hence $\lambda_\tau \geq \sqrt R \cdot ( 1/2 - 1/(2\sqrt 2) )$ (in particular, $\lambda_\tau$ is non-negative).
	Together with the calculation
	\begin{align*}
		\left[ \left( \frac{R}{4} \right)^{1/2}  + \tau_0 \left( \frac{R\zeta}{16n} \right)^{1/2}  \right]^{2} + \left( \frac{R\zeta}{16n} \right) \sum_{1 \leq \jabs \leq \mopt} \frac{\gamma_j^2}{n} \leq \frac{R}{2} + \frac{R\zeta}{16n} \gamma_\mopt^2 \frac{2\mopt + 1}{n} \leq R 
	\end{align*}
	this shows that $\lambda_\tau \in \Theta_\gamma^R$ for every $\tau \in \{  \pm 1 \}^{2\mopt + 1}$.
	
	We now derive a reduction scheme which holds for an arbitrary estimator $\widetilde \lambda$ of $\lambda$.
	For this purpose, denote $\boldsymbol X = (X_1,\ldots, X_n)$, $\boldsymbol Y = (Y_1,\ldots, Y_n)$, and by $\P_\tau^{\boldsymbol Y| \boldsymbol X}$ the conditional distribution of $\boldsymbol Y$ given $\boldsymbol X$ when the true regression function is $\lambda_\tau$.
	By $\E_\tau[\cdot | \boldsymbol X]$ we denote the corresponding conditional expectation operator.
	Then the following reduction scheme holds
	\begin{align}
	\sup_{\lambda \in \Theta_\gamma^R} \E \Vert \widetilde \lambda - \lambda \Vert^2 &\geq \frac{1}{2^{2\mopt + 1}} \sum_{\tau \in \{ \pm 1 \}^{2\mopt +1} } \sum_{0 \leq \vert j \vert \leq \mopt}  \E \E_\tau [ \vert \thetatilde_j - \theta_{\tau j} \vert^2 | \Xb ] \notag \\
	&\hspace{-6em} = \frac{1}{2^{2\mopt + 1}} \sum_{0 \leq \vert j \vert \leq \mopt} \sum_{\tau \in \{ \pm 1 \}^{2\mopt +1} } \frac{1}{2} \{ \E  \E_\tau [ \vert \thetatilde_j - \theta_{\tau j} \vert^2 | \Xb ] + \E \E_{\tau^{j}} [ \vert \thetatilde_j - \theta_{\tau^j j} \vert^2 | \Xb ] \} \label{PR:reduction}
	\end{align}
	where $\thetatilde_j$, $\theta_{\tau j}$ are the coefficients of $\lambdatilde$ and $\lambda_\tau$ corresponding to the basis function $\phi_j$, respectively, and for $\tau \in \{ \pm 1 \}^{2\mopt + 1}$ the element $\tau^{j} \in \{ \pm 1 \}^{2\mopt + 1}$ is defined by $\tau_k^{j}=\tau_k$ for $k \neq j$ and $\tau_j^{j} = - \tau_j$.
	Consider the Hellinger affinity defined through
	\begin{equation*}
	  \rho(\P_\tau^{\Yb|\Xb}, \P_{\tau^{j}}^{\Yb|\Xb}) = \int \sqrt{d\P_\tau^{\Yb|\Xb} d\P_{\tau^{j}}^{\Yb|\Xb} }.
	\end{equation*}
	We have
	\begin{align*}
	\rho(\P_\tau^{\Yb|\Xb}, \P_{\tau^{(j)}}^{\Yb|\Xb}) &\leq \int \frac{\vert \thetatilde_j - \theta_{\tau j} \vert}{\vert \theta_{\tau j} - \theta_{\tau^j j} \vert} \sqrt{d\P_\tau^{\Yb|\Xb} d\P_{\tau^{j}}^{\Yb|\Xb} }  + \int \frac{\vert \thetatilde_j - \theta_{\tau^j j} \vert}{\vert \theta_{\tau j} - \theta_{\tau^j j} \vert} \sqrt{d\P_\tau^{\Yb|\Xb} d\P_{\tau^{(j)}}^{\Yb|\Xb} } \\
	& \leq \left( \int \frac{\vert \thetatilde_j - \theta_{\tau j} \vert^2}{\vert \theta_{\tau j} - \theta_{\tau^{j} j} \vert^2} d\P_\tau^{\Yb | \Xb} \right)^{1/2} + \left( \int \frac{\vert \thetatilde_j - \theta_{\tau^j j} \vert^2}{\vert \theta_{\tau j} - \theta_{\tau^j j} \vert^2} d\P_{\tau^{j}}^{\Yb | \Xb} \right)^{1/2}.
	\end{align*}
	By means of the estimate $(a+b)^2 \leq 2a^2 + 2b^2$ we obtain
	\begin{equation}\label{eq:hell:aff}
	\frac{1}{2} \vert \theta_{\tau j} - \theta_{\tau^j j} \vert^2 \rho^2 (\P_\tau^{\Yb | \Xb}, \P_{\tau^{j}}^{\Yb | \Xb}) \leq \E_\tau [ \vert \thetatilde_j - \theta_{\tau j} \vert^2 | \Xb ] + \E_{\tau^{j}}[ \vert \thetatilde_j - \theta_{\tau^j j} \vert^2 | \Xb ].
	\end{equation}
	For any $\tau \in \{ \pm 1 \}^{2\mopt + 1}$, let us denote by $\P_\tau^{Y_i|\Xb}$ the marginal distribution of $Y_i$ given $\Xb$ (since the distribution of $Y_i$ given $\Xb$ depends on $X_i$ only, we could equally write $\P_\tau^{Y_i|X_i}$).
	Formula~(5) from~\cite{roos2003improvements} allows us to bound the total variation distance $\mathrm{TV}(\P_\tau^{Y_i | \Xb}, \P_{\tau^{j}}^{Y_i | \Xb})$ between $\P_\tau^{Y_i | \Xb}$ and $\P_{\tau^{j}}^{Y_i | \Xb}$ as
	\begin{equation*}
		\mathrm{TV}(\P_\tau^{Y_i | \Xb}, \P_{\tau^{j}}^{Y_i | \Xb}) \leq \vert \lambda_\tau(X_i) - \lambda_{\tau^j}(X_i) \vert \leq \vert 2\sqrt{2}(R\zeta/(16n))^{1/2}  \vert \leq \left( \frac{R\zeta}{2n} \right)^{1/2} \leq \frac{1}{\sqrt n}
	\end{equation*}
	due to the definition of $\zeta$.
	Recall the definition $H^2(\P, \Q) = \int [\sqrt{d\P} - \sqrt{d\Q}]^2$ of the squared Helliger distance between two probability measures $\P$ and $\Q$.
	By formula (2.20) from~\cite{tsybakov2009introduction} it follows that
	\begin{equation*}
		H^2(\P_\tau^{Y_i | \Xb}, \P_{\tau^{j}}^{Y_i | \Xb}) \leq \frac{1}{n}.
	\end{equation*}
%	Recall the definition of the Hellinger distance,
%	\begin{equation*}
%	H(\P_\tau^{\Yb | \Xb}, \P_{\tau^{j}}^{\Yb | \Xb}) \defeq \left( \int \left[ \sqrt{\P_\tau^{\Yb | \Xb}} - \sqrt{\P_{\tau^{j}}^{\Yb | \Xb}} \right]^2 \right)^{1/2}.
%	\end{equation*}
%	Let $N_i$ be a PPP on $[0,1]$ with constant intensity equal to $\lambda(X_i)$.
%	Consider the transformation which maps the point process $N_i$ to $Y_i = N_i([0, 1])$.
%	Using Lemma~3.3.13 from~\cite{reiss1989approximate} and Lemma~3.2.1 from~\cite{reiss1993course} we can conclude
%	\begin{align*}
%	H^2(\P_\tau^{Y_i | \Xb}, \P_{\tau^{j}}^{Y_i | \Xb}) \leq H^2(\P_\tau^{N_i | \Xb}, \P_{\tau^{j}}^{N_i | \Xb}) &\leq \int_{0}^{1} (\sqrt{\lambda_\tau(X_i)} - \sqrt{\lambda_{\tau^{j}}(X_i)})^2 ds\\
%	&= \frac{\vert \lambda_\tau(X_i) - \lambda_{\tau^{j}}(X_i) \vert^2}{(\sqrt{\lambda_\tau (X_i)} + \sqrt{\lambda_{\tau^{j}}(X_i)})^2} \\
%	&\leq \frac{\zeta \sqrt R}{8n\delta} \leq \frac{1}{n}.
%	\end{align*}
	Since $Y_1,\ldots,Y_n$ are independent conditionally on $X_1,\ldots,X_n$ we obtain by Lemma~3.3.10~(i) from~\cite{reiss1993course} that
	\begin{equation*}
	H^2(\P_\tau^{\Yb | \Xb}, \P_{\tau^{j}}^{\Yb | \Xb}) \leq \sum_{i=1}^{n} H^2(\P_\tau^{Y_i | \Xb}, \P_{\tau^{j}}^{Y_i | \Xb}) \leq 1.
	\end{equation*}
	Hence the relation $\rho(\P_\tau^{\Yb | \Xb}, \P_{\tau^{j}}^{\Yb | \Xb}) = 1- H^2(\P_\tau^{\Yb | \Xb}, \P_{\tau^{j}}^{\Yb | \Xb})/2$ (see~\cite{tsybakov2009introduction}, p.~87) implies $\rho(\P_\tau^{\Yb | \Xb}, \P_{\tau^{j}}^{\Yb | \Xb}) \geq 1/2$.
	Putting this estimate into the reduction scheme~\eqref{PR:reduction} finally yields using~\eqref{eq:hell:aff} that
	\begin{align*}
	\sup_{\lambda \in \Theta_\gamma^R} \E \Vert \widetilde \lambda - \lambda \Vert^2 &\geq \frac{1}{2^{2\mopt + 1}} \sum_{\tau \in \{ \pm 1 \}^{2\mopt + 1} } \sum_{0 \leq \vert j \vert \leq \mopt} \frac{1}{2} \E [ \E_\tau [ \vert \thetatilde_j - \theta_{\tau j} \vert^2 | \Xb] + \E_{\tau^{j}} [ \vert \thetatilde_j - \theta_{\tau^j j} \vert^2 | \Xb] ] \\
	& \geq \frac{1}{16} \sum_{0 \leq \vert j \vert \leq \mopt} \vert \theta_{\tau j} - \theta_{\tau^j j} \vert^2 = \frac{ R\zeta}{64} \sum_{0 \leq \vert j\vert \leq \mopt} \frac{1}{n} = \frac{R\zeta}{64} \cdot \frac{2\mopt + 1}{n}.
	\end{align*}
	Since the last estimate holds for arbitrary $\widetilde \lambda$, we obtain the claim assertion by means of Assumption~\ref{PR:it:C2}.

\subsection{Proof of Theorem~\ref{thm:adaptive:xi:known}}

	Note that the identity $\Upsilon_n(f) = \Vert \lambdahat_n - f \Vert^2 - \Vert \lambdahat_n \Vert^2$ holds for all $f \in L^2$.
	Hence $\lambdahat_\mf = \argmin_{f \in \Sc_\mf} \Upsilon_n(f)$ for all $\mf \in \Mc_n$, and exploiting the definition of $\mftilde$ yields for all $\mf \in \Mc_n$ that
	\begin{equation*}
		\Upsilon_n(\lambdahat_\mftilde) + \pen(\mftilde) \leq \Upsilon_n(\lambdahat_\mf) + \pen(\mf) \leq \Upsilon_n(\lambda_\mf) + \pen(\mf)
	\end{equation*}
	where $\lambda_\mf = \sum_{\eta \in \Ic_\mf} \theta_\eta \phi_\eta$ is the projection of $\lambda$ on the finite-dimensional space $\Sc_\mf$.
	In a similar manner, let us denote by $\lambda_n$ the projection of $\lambda$ on the space $\Sc_n$, i.e., $\lambda_n = \sum_{\eta \in \Ic_n} \theta_\eta \phi_n$.
	By definition of the contrast and some algebra we obtain
	\begin{equation*}
		\Vert \lambdahat_\mftilde - \lambda \Vert^2 \leq \Vert \lambda_\mf - \lambda \Vert^2 + 2 \langle \lambdahat_n - \lambda_n, \lambdahat_\mftilde - \lambda_\mf \rangle + \pen(\mf) - \pen(\mftilde).
	\end{equation*}
	Setting $\thetatilde_\eta = \frac{1}{n} \sum_{i=1}^{n} \lambda(X_i) \phi_\eta(X_i)$ and $\lambdatilde_n = \sum_{\eta \in \Ic_n} \thetatilde_\eta \phi_\eta$ we obtain
	\begin{equation*}
		\Vert \lambdahat_\mftilde - \lambda \Vert^2 \leq \Vert \lambda_\mf - \lambda \Vert^2 + 2 \langle \Thetahat_n, \lambdahat_\mftilde - \lambda_\mf \rangle + 2 \langle \Thetatilde_n, \lambdahat_\mftilde - \lambda_\mf \rangle + \pen(\mf) - \pen(\mftilde)
	\end{equation*}
	where $\Thetahat_n = \lambdahat_{n} - \lambdatilde_{n}$ and $\Thetatilde_n = \lambdatilde_{n} - \lambda_{n}$.
	Set $\Bc_\mf = \{ \lambda \in \Sc_\mf : \Vert \lambda \Vert \leq 1 \}$.
	Using the estimate $2xy \leq \tau x^2 + \tau^{-1} y^2$ for positive $\tau$ (below  we specialize with $\tau = 1/8$) we conclude
	\begin{align*}
		\Vert \lambdahat_{\mftilde} - \lambda \Vert^2 &\leq \Vert \lambda_\mf - \lambda \Vert^2 + 2 \tau \Vert \lambdahat_\mftilde - \lambda_\mf \Vert^2 + \tau^{-1} \sup_{t \in \Bc_{\mf \vee \mftilde} } \vert \langle \Thetahat_n, t \rangle \vert^2  + \tau^{-1} \sup_{t \in \Bc_{\mf \vee \mftilde}} \vert \langle \Thetatilde_n, t \rangle \vert^2\\
		&\hspace{1em} + \pen(\mf) - \pen(\mftilde)
	\end{align*}
	(here, $\mf \vee \mftilde$ denotes the maximal model of $\mf$ and $\mftilde$ such that $\Sc_{\mf \vee \mftilde} = \Sc_\mf \cup \Sc_\mftilde$; this model exists thanks to Assumption~\ref{ass:model:adaptive}).
	Specializing with $\tau = 1/8$ we conclude that for each model $\mf \in \Mc_n$
	\begin{align*}
		\Vert \widehat \lambda_{\mftilde} - \lambda \Vert^2 &\leq 3 \Vert \lambda-\lambda_\mf \Vert^2 + 16 \sup_{t \in \Bc_{\mf \vee \mftilde}} \vert \langle \Thetahat_n, t \rangle \vert^2 + 16 \sup_{t \in \Bc_{\mf \vee \mftilde}} \vert \langle \Thetatilde_n, t \rangle \vert^2 + 2\pen(\mf) - 2 \pen(\mftilde)\\
		&\leq 3 \Vert \lambda-\lambda_\mf \Vert^2 + 16 \left( \sup_{t \in \Bc_{\mf \vee \mftilde}} \vert \langle \Thetahat_n, t \rangle \vert^2 - 50 \mu \cdot \frac{\Phi^2\D_{\mf \vee \mftilde} \log(n+2)}{n} \right)_+\\
		&\hspace{1em}+ 16 \left( \sup_{t \in \Bc_{\mf \vee \mftilde}} \vert \langle \Thetatilde_n, t \rangle \vert^2 - 3 \mu \cdot \frac{\Phi^2\D_{\mf \vee \mftilde}}{n} \right)_+\\
		&\hspace{1em}+800\mu \cdot \frac{\Phi^2\D_{\mf \vee \mftilde} \log(n+2)}{n} + 48 \mu \cdot \frac{\Phi^2\D_{\mf \vee \mftilde}}{n} + 2 \pen(\mf) - 2 \pen(\mftilde).
	\end{align*}
	By definition of the penalty, the estimate $\D_{\mf \vee \mftilde} \leq \D_\mf + \D_{\mftilde}$ and roughly bounding the brackets $(\ldots)_+$ by summing over all potential models $\mf \in \Mc_n$, we have
	\begin{align*}
		\Vert \widehat \lambda_{\mftilde} - \lambda \Vert^2 &\leq 3 \Vert \lambda - \lambda_\mf \Vert^2 + 16 \sum_{\mf' \in \Mc_n} \left( \sup_{t \in \Bc_{\mf'}} \vert \langle \Thetahat_n, t \rangle \vert^2 - 50 \mu \cdot \frac{\Phi^2\D_{\mf'} \log(n+2)}{n} \right)_+\\
		&+ 16 \sum_{\mf' \in \Mc_n} \left( \sup_{t \in \Bc_\mf'} \vert \langle \Thetatilde_n, t \rangle \vert^2 - 3 \mu \cdot \frac{\Phi^2\D_{\mf'}}{n} \right)_+ + 4 \pen(\mf).
	\end{align*}
	Taking expectations and into account that the last estimate holds for each $\mf \in \Mc_n$, we obtain
	\begin{align}\label{eq:adap:xi:known:sums}
		\E \Vert \lambdahat_\mftilde - \lambda \Vert^2 &\leq \min_{\mf \in \Mc_n} \{ 3 \Vert \lambda - \lambda_\mf \Vert^2 + 4\pen(\mf) \}\notag\\
		&+ 16 \sum_{\mf' \in \Mc_n} \E \left[ \left( \sup_{t \in \Bc_{\mf'}} \vert \langle \Thetahat_n, t \rangle \vert^2 - 50 \mu \cdot \frac{\Phi^2 \D_{\mf'} \log(n+2)}{n} \right)_+ \right]\notag\\
		&+ 16 \sum_{\mf' \in \Mc_n} \E \left[ \left( \sup_{t \in \Bc_{\mf'}} \vert \langle \Thetatilde_n, t \rangle \vert^2 - 3 \mu \cdot \frac{\Phi^2 \D_{\mf'}}{n} \right)_+ \right]\notag\\
		&\eqdef \min_{\mf \in \Mc_n} \{ 3 \Vert \lambda - \lambda_\mf \Vert^2 + 4\pen(\mf) \} + 16 \sum_{\mf' \in \Mc_n} E_{\mf' 1} + 16 \sum_{\mf' \in \Mc_n} E_{\mf' 2}
	\end{align}
	We now use Lemmata~\ref{conc:Theta:tilde} and~\ref{conc:Theta:hat} to bound the terms $E_{\mf'1}$ and $E_{\mf'2}$ which yields
	\begin{align*}
		E_{\mf'1} &\leq K_1' \left\lbrace \frac{\D_{\mf'}}{n} \exp(-2 \log(n+2)) + \frac{\D_{\mf'}}{n^2} \exp(-K_2'\sqrt{n}) \right\rbrace, \quad \text{and}\\
		E_{\mf'2} &\leq K_1 \left\lbrace \frac{1}{n} \exp(-K_2 \D_{\mf'}) + \frac{\D_{\mf'}}{n^2} \exp(-K_3 \sqrt n) \right\rbrace.
	\end{align*}
	Putting these estimates into~\eqref{eq:adap:xi:known:sums}, using $\D_{\mf'} \leq c_\mf n$ for $\mf' \in \Mc_n$ and $\vert \Mc_n \vert \leq c_{\Mc} n$ (which hold due to Assumption~\ref{ass:model:adaptive}), we obtain
	\begin{equation*}
		\E \Vert \lambdahat_\mftilde - \lambda \Vert^2 \lesssim \min_{\mf \in \Mc_n} \max \{ \Vert \lambda_\mf - \lambda \Vert^2, \pen(\mf) \} + \frac{1}{n}.
	\end{equation*}

\subsection{Proof of Theorem~\ref{thm:adaptive:xi:unknown}}

	Let us introduce the event $\Xi = \left\lbrace \left|  \frac{\Vert \lambdahat_\Pi \Vert_\infty \vee 1}{\Vert \lambda \Vert_\infty \vee 1} - 1 \right|  < \frac{3}{4} \right\rbrace$.
	It is readily verified that on $\Xi$ it holds that
	\begin{equation*}
		\Vert \lambda \Vert_\infty \vee 1 \leq 4 ( \Vert \lambdahat_\Pi \Vert_\infty \vee 1 ) \quad \text{and} \quad \Vert \lambdahat_\Pi \Vert_\infty \vee 1 \leq \frac{7}{4} ( \Vert \lambda \Vert_\infty \vee 1 ). 
	\end{equation*}
	These estimates will be used below without further reference.
	We consider the decomposition
	\begin{equation*}
		\E \Vert \lambdahat_\mfhat - \lambda \Vert^2 = \E \Vert \lambdahat_\mfhat - \lambda \Vert^2 \1_\Xi + \E \Vert \lambdahat_\mfhat - \lambda \Vert^2 \1_{\Xi^\complement} \eqdef T_1 + T_2.
	\end{equation*}
	\noindent \emph{Upper bound for $T_1$}: In analogy to the proof of Theorem~\ref{thm:adaptive:xi:known} one can derive (in the following, all the appearing quantities are defined as in the proof of Theorem~\ref{thm:adaptive:xi:known})
%	\begin{align*}
%		\Vert \lambdahat_\mfhat - \lambda \Vert^2 &\leq \Vert \lambda_\mf - \lambda \Vert^2 + 2 \tau \Vert \lambdahat_\mfhat - \lambda_\mf \Vert^2 + \tau^{-1} \sup \vert \langle \Thetahat_n, t\rangle \vert^2\\
%		&\hspace{1em}+ \tau^{-1} \sup \vert \langle \Thetatilde,t \rangle \vert^2 + \penhat(\mf) - \penhat(\mfhat)
%	\end{align*}
%	for all $\mf \in \Mc_n$, and all the appearing quantities are defined exactly as in the proof of Theorem~\ref{thm:adaptive:xi:known}.
%	Using the same arguments as in that proof, one obtains by specializing with $\tau = \frac{1}{8}$ and setting $\mu = {1 \vee \xi^2}$ (recall that $\xi$ satisfies $\Vert \lambda \Vert_\infty \leq \xi$) that
	\begin{align*}
		\Vert \lambdahat_\mfhat - \lambda \Vert^2 &\leq 3 \Vert \lambda_\mf - \lambda \Vert^2 + 16 \sup_{t \in \Bc_\mf} \vert \langle \Thetahat_n, t \rangle \vert^2 + 16 \sup_{t \in \Bc_\mf} \vert \langle \Thetatilde_n, t \rangle \vert^2 + 2 \penhat(\mf) - 2\penhat(\mfhat)\\
		&\leq 3 \Vert \lambda_\mf - \lambda \Vert^2 + 16 \left( \sup_{t \in \Bc_\mf} \vert \langle \Thetahat_n, t\rangle \vert^2 - 50\mu \cdot \frac{\Phi^2\D_{\mf \vee \mfhat} \log (n+2)}{n} \right)_+\\
		&\hspace{1em}+ 16 \left( \sup_{t \in \Bc_\mf} \vert \langle \Thetatilde_n, t\rangle \vert^2 - 3 \mu \cdot \frac{\Phi^2\D_{\mf \vee \mfhat}}{n} \right)_+\\
		&\hspace{1em}+ 800 \mu \cdot \frac{\Phi^2\D_{\mf \vee \mfhat} \log(n+2) }{n} + 48 \mu \cdot \frac{\Phi^2\D_{\mf \vee \mfhat}}{n} + 2 \penhat(\mf) - 2 \penhat(\mfhat). 
	\end{align*}
	By definition of $\Xi$ and the random penalty function, we obtain (note that $\D_{\mf \vee \mfhat} \leq \D_\mf + \D_\mfhat$)
	\begin{align*}
		\Vert \lambdahat_\mfhat - \lambda \Vert^2 \1_\Xi &\leq 3 \Vert \lambda_\mf - \lambda \Vert^2 + 16 \left( \sup_{t \in \Bc_\mf} \vert \langle \Thetahat_n, t\rangle \vert^2 - 50\mu \cdot \frac{\Phi^2\D_{\mf \vee \mfhat} \log (n+2)}{n} \right)_+\\
		&+ 16 \left( \sup_{t \in \Bc_\mf} \vert \langle \Thetatilde_n, t\rangle \vert^2 - 3 \mu \cdot \frac{\Phi^2\D_{\mf \vee \mfhat}}{n} \right)_+ + 100 \pen(\mf).
	\end{align*} 
	Bounding the terms in the brackets $(\ldots)_+$ by summing over all admissible models $\mf \in \Mc_n$ and taking expectations on both sides yields
	\begin{align*}
		\E \Vert \lambdahat_\mfhat - \lambda \Vert^2 \1_\Xi &\leq 3 \Vert \lambda_\mf - \lambda \Vert^2 + 100 \pen(\mf) \\
		&\hspace{1em}+ 16 \sum_{\mf' \in \Mc_n} \E \left[ \left( \sup_{t \in \Bc_{\mf'}} \vert \langle \Thetahat_n, t \rangle \vert^2 - 50 \mu \cdot \frac{\Phi^2 \D_{\mf'} \log(n+2)}{n} \right)_+  \right]  \\
		&\hspace{1em}+ 16 \sum_{\mf' \in \Mc_n} \E \left[ \left( \sup_{t \in \Bc_{\mf'}} \vert \langle \Thetatilde_n, t \rangle \vert^2 - 3 \mu \cdot \frac{\Phi^2 \D_{\mf'}}{n} \right)_+ \right] 
	\end{align*}
	Applying Lemmata~\ref{conc:Theta:tilde} and~\ref{conc:Theta:hat} as in the proof of Theorem~\ref{thm:adaptive:xi:known} then implies
	\begin{equation*}
		\E \Vert \lambdahat_\mfhat - \lambda \Vert^2 \1_\Xi \lesssim \inf_{\mf \in \Mc_n} \max \{ \Vert \lambda_\mf - \lambda \Vert^2, \pen(\mf)  \} + \frac{1}{n}. % + \exp(-\kappa \sqrt n).
	\end{equation*}
	
	\noindent \emph{Upper bound for $T_2$}: 
	First, take note of the estimate
	\begin{equation*}
		\E \Vert \lambdahat_\mfhat - \lambda \Vert^2 \1_{\Xi^\complement}  \leq 2 \E  \Vert \lambdahat_\mfhat \Vert^2 \1_{\Xi^\complement} + 2 \E \Vert \lambda \Vert^2 \1_{\Xi^\complement}.
	\end{equation*}
	We have
	\begin{equation*}
		\Vert \lambdahat_\mfhat \Vert^2 = \sum_{\eta \in \Ic_\mfhat} \thetahat_\eta^2 \leq \frac{1}{n} \sum_{i=1}^{n} Y_i^2 \sum_{\eta \in \Ic_\mfhat} \phi_\eta^2(X_i) \leq \frac{\Phi^2 \D_\mfhat}{n} \sum_{i=1}^{n} Y_i^2.
	\end{equation*}
	Hence, by means of the Cauchy-Schwarz inequality and the fact that $\D_\mfhat \leq c_\mf n$ due to Assumption~\ref{ass:model:adaptive},
	\begin{align*}
		\E \Vert \lambdahat_\mfhat - \lambda \Vert^2 \1_{\Xi^\complement}  &\leq 2\Phi^2 c_\mf \left( \E \left[ \left( \sum_{i=1}^n Y_i^2\right) ^2\right] \right) ^{1/2} \P (\Xi^\complement)^{1/2}  + 2 \Vert \lambda \Vert^2 \P (\Xi^\complement)\\
		&\leq 2\Phi^2 c_\mf \left( \E \left[ n\sum_{i=1}^n Y_i^4\right]  \right) ^{1/2}  \P (\Xi^\complement)^{1/2}  + 2 \Vert \lambda \Vert^2 \P (\Xi^\complement)\\
		&\leq 2 \Phi^2c_\mf n T_4(\Vert \lambda \Vert_\infty)^{1/2} \P (\Xi^\complement)^{1/2}  + 2 \Vert \lambda \Vert^2 \P (\Xi^\complement)
	\end{align*}
	where $T_4$ is the fourth Touchard polynomial $T_4$.
	By the above estimates it suffices to show that $\P (\Xi^\complement) \lesssim n^{-4}$.
	Note that we have
	\begin{equation}\label{eq:ex:rev:tri}
		\vert \Vert \lambdahat_\Pi \Vert_\infty - \Vert \lambda \Vert_\infty \vert \leq \Vert \lambdahat_\Pi - \lambda_\Pi \Vert_\infty + \Vert \lambda_\Pi - \lambda \Vert_\infty \leq \Vert \lambdahat_\Pi - \lambda \Vert_\infty + \frac{1}{4} \Vert \lambda \Vert_\infty
	\end{equation}
	where the last estimate holds due to Assumption~\ref{ass:Pi1}.
	Putting $\phi_j = \frac{1}{\sqrt{\P(\X_j)}} \1_{\X_j}$ we have
	\begin{align*}
		\Vert \lambdahat_\Pi - \lambda_\Pi \Vert_\infty &= \sup_{1 \leq j \leq M} \Vert (\lambdahat_\Pi - \lambda_\Pi) \1_{\X_j} \Vert_\infty\\
		&= \sup_{1 \leq j \leq M} \P(\X_j)^{-1/2} \Vert (\lambdahat_\Pi - \lambda_\Pi) \1_{\X_j} \Vert\\
		&= \sup_{1 \leq j \leq M} \Vert (\lambdahat_\Pi - \lambda_\Pi) \phi_j \Vert\\
			&\leq \sup_{1 \leq j \leq M} \vert \langle \lambdahat_\Pi - \E[\lambdahat_\Pi | X] ,\phi_j \rangle \vert + \sup_{1 \leq j \leq M} \vert \langle \E[\lambdahat_\Pi | X] - \lambda_\Pi,\phi_j \rangle \vert\\
		&= \sup_{1 \leq j \leq M} \vert \nu(\phi_j) \vert + \sup_{1 \leq j \leq M} \vert \nutilde(\phi_j) \vert
	\end{align*}
	where $\nu(\phi_j) = \langle \lambdahat_\Pi - \E[\lambdahat_\Pi | X] ,\phi_j \rangle$ and $\nutilde(\phi_j) = \langle \E[\lambdahat_\Pi | X] - \lambda_\Pi,\phi_j \rangle$.
	Using~\eqref{eq:ex:rev:tri} and the estimate $\vert a \vee 1 - b \vee 1 \vert \leq \vert a - b\vert$, we obtain
	\begin{align*}
		\P (\Xi^\complement) &= \P ( \vert \Vert \lambdahat_\Pi \Vert_\infty \vee 1 - \Vert \lambda \Vert_\infty \vee 1 \vert \geq 3/4 \cdot (\Vert \lambda \Vert_\infty \vee 1) )\\
		&\leq \P (\Vert \widehat \lambda_\Pi - \lambda_\Pi \Vert_\infty \geq 1/2 \cdot (\Vert \lambda \Vert_\infty \vee 1) )\\
		&\leq \P (\sup_{1 \leq j \leq M} \vert \nu(\phi_j) \vert \geq 1/4 \cdot (\Vert \lambda \Vert_\infty \vee 1)) + \P (\sup_{1 \leq j \leq M} \vert \nutilde(\phi_j) \vert \geq 1/4 \cdot (\Vert \lambda \Vert_\infty \vee 1))\\
		&\leq \sum_{j=1}^{M} \bigg\{ \P(\nu(\phi_j) \geq \xi) + \P(-\nu(\phi_j) \geq \xi)	+ \P(\nutilde(\phi_j) \geq \xi)  + \P(-\nutilde(\phi_j) \geq \xi) \bigg\}
		\end{align*}
	where $\xi = (\Vert \lambda \Vert_\infty \vee 1)/4$.
	Note that $\Vert \phi_j \Vert=1$ and $\Vert \phi_j \Vert_\infty = \P(\X_j)^{-1/2}$.
	By application of Proposition~\ref{PR:prop:bernstein:PPP} we obtain putting $\underbar p = \inf_{j = 1,\ldots,M} \P(\X_j)$
	\begin{align*}
		\P (\pm \nu(\phi_j) \geq \xi) &\leq \exp \left( - \frac{n\xi^2}{2 \Vert \phi_j \Vert_\infty^2 \Vert \lambda \Vert_\infty + 2/3 \xi \Vert \phi_j \Vert_\infty} \right)\\
		&\leq \exp \left( -\frac{1}{4} \left( \frac{n\xi^2}{\Vert \phi_j \Vert_\infty^2 (\Vert \lambda \Vert_\infty \vee 1)} \wedge \frac{3n\xi}{\Vert \phi_j \Vert_\infty} \right)  \right)\\
		&\leq \exp \left(- \frac{n (\Vert \lambda \Vert_\infty \vee 1)}{64} \cdot \underbar p \right).
	\end{align*}
	Analogously, exploiting Proposition~\ref{PR:prop:bernstein}, we get
	\begin{align*}
		\P (\pm \nutilde(\phi_j) \geq \xi) &\leq \exp \left( - \frac{n}{64} \cdot \underbar p \right),
	\end{align*}
	and hence
	\begin{equation*}
		\P (\Xi^\complement) \leq 4 M \exp \left( - \frac{n}{64} \cdot \underbar p \right) 
	\end{equation*}
	Assumption~\ref{ass:Pi2} finally implies $\P (\Xi^\complement) \leq \frac{c_\Pi n}{80 \log n} \cdot n^{-5} \lesssim n^{-4}$.

\section{Proofs of Section~\ref{sec:dependent}}

\subsection{Proof of Proposition~\ref{prop:dep:upper}}

For the proof of Proposition~\ref{prop:dep:upper} we need the following lemma which is inspired by statement~(i) of Theorem~2.1 in~\cite{viennet1997inequalities} (see also Lemma~4.1 in \cite{asin2016adaptive_a} where this lemma was also exploited).

\begin{lemma}\label{lem:var:bound:beta} % entspricht Lemma 4.1 aus Asin/Johannes
	Let $(X_i)_{i \in \Z}$ be a strictly stationary absolutely regular process with $\beta$-mixing sequence $(\beta_k)_{k \in \N_0}$.
	Then there exists a sequence of measurable functions $b_k: \X \to [0,1]$ with $b_0 \equiv 1$, $0 \leq b_k \leq 1$, $\E b_k(\xi_0) = \beta_k$ such that for any $g \in L^2$ and any $n \in \N$
	\begin{equation*}
		\Var \left( \sum_{i=1}^n g(X_i) \right) \leq 4n \E \left[ \left( \sum_{k=0}^{n} b_k\right)  g^2(X_0)\right].
	\end{equation*} 
\end{lemma}

\begin{proof}[Proof of the Lemma]
	% nur fuer mich der Vollstaendigkeit halber drin, spaeter raus
	Thanks to Lemma~4.1 in \cite{viennet1997inequalities}, there exist two functions $b_k'$ and $b_k''$ from $\X$ to $[0,1]$ such that $\E b_k' = \E b_k''=\beta_k$ and
	\begin{equation*}
		\Cov(g(X_0), g(X_k)) \leq 2 \E [b_k' g^2]^{1/2} \E [b_k'' g^2]^{1/2}.
	\end{equation*}
	Thus
	\begin{align*}
		\Var \left( \sum_{i=1}^n g(X_i) \right) &\leq 2 \sum_{k=0}^{n} (n-k) \vert\Cov(g(X_0), g(X_k))\vert\\
		&\leq 4n \sum_{k=0}^n \E [b_k' g^2]^{1/2} \E [b_k'' g^2]^{1/2}\\
		&\leq 4n \sum_{k=0}^{n} \E [(b_k' + b_k'')g^2/2]
	\end{align*}
	which finishes the proof by defining $b_k=(b_k' + b_k'')/2$.
\end{proof}

	We come now to the proof of Proposition~\ref{prop:dep:upper}.
	As in the proof of Proposition~\ref{prop:ind:upper}, we use the bias-variance decomposition
	\begin{equation*}
		\E \Vert \lambdahat_\mf - \lambda \Vert^2 = \Vert \lambda_\mf - \lambda \Vert^2 + \E \Vert \lambdahat_\mf - \lambda_\mf \Vert^2.
	\end{equation*}
	For the variance term, we obtain exploiting Assumption~\ref{ass:model} and Lemma~\ref{lem:var:bound:beta}
	\begin{align*}
		\E \Vert \lambdahat_\mf - \lambda_\mf \Vert^2 &= \sum_{\eta \in \Ic_\mf} \Var ( \thetahat_\eta ) = \frac{1}{n^2} \sum_{\eta \in \Ic_\mf} \Var \left( \sum_{i=1}^{n} Y_i \phi_\eta(X_i) \right)\\
		&\hspace{-1em}= \frac{1}{n^2} \sum_{\eta \in \Ic_\mf} \Var \left(\sum_{i=1}^{n} \E [Y_i \phi_\eta(X_i) | X_i]  \right) + \frac{1}{n^2} \sum_{\eta \in \Ic_\mf} \E \left[ \Var \left( \sum_{i=1}^{n} Y_i \phi_\eta(X_i) \Big \vert X_i \right) \right]\\
		&\hspace{-1em}= \frac{1}{n^2} \sum_{\eta \in \Ic_\mf} \Var \left(\sum_{i=1}^{n} \lambda(X_i) \phi_\eta(X_i) \right) + \frac{1}{n^2} \sum_{\eta \in \Ic_\mf} \E \left[ \sum_{i=1}^{n} \phi_\eta^2(X_i) \lambda(X_i) \right]\\
		&\hspace{-1em}\leq \frac{4}{n} \sum_{\eta \in \Ic_\mf} \E \left[ \left( \sum_{k=0}^{n} b_k(X_0) \right) \phi_\eta^2(X_0) \lambda^2(X_0) \right]  + \frac{\Phi^2\D_\mf}{n} \E [\lambda(X_1)]\\
		&\hspace{-1em} \leq 4 \left( \sum_{k=0}^{n} \beta_k\right)  \Vert \lambda \Vert_\infty^2 \Phi^2\cdot \frac{\D_\mf}{n}  + \frac{\D_\mf}{n} \cdot \Phi^2 \Vert \lambda \Vert_1\\
		&= \frac{\Phi^2 \D_\mf}{n} \left[ \Vert \lambda \Vert_1 + 4 \Vert \lambda \Vert_\infty^2 \left( \sum_{k=0}^{n} \beta_k \right) \right].
	\end{align*}

\subsection{Proof of Theorem~\ref{thm:adaptive:dep}}

	We put $q_n = \lceil \sqrt{n/2} \rceil$.
	In the following we assume for the sake of simplicity that $n = 2 p_n q_n$ with $p_n$ being an integer.
	For $\ell = 0, \ldots,p_n-1$ put $A_\ell = (X_{2\ell q_n + 1},\ldots,X_{(2\ell+1)q_n})$ and $B_\ell = (X_{(2\ell +1)q_n+1},\ldots,X_{(2\ell + 2)q_n})$.
	Exploiting a construction given in~\cite{viennet1997inequalities} on the basis of Berpee's coupling lemma, we can create $A_\ell^\ast$ for $\ell = 0,\ldots,p_n-1$ such that
	\begin{itemize}%[label=$\triangleright$]
		\item $A_\ell$ and $A_\ell^\ast$ have the same distribution,
		\item $A_\ell^\ast$ and $A_{\ell'}^\ast$ are independent if $\ell \neq \ell'$, and
		\item $\P (A_\ell \neq A_\ell^\ast) \leq \beta_{q_n}$.
	\end{itemize}
	In the same fashion, one can build $B_\ell^\ast$ for $\ell = 0,\ldots,p_n-1$.
 	We now define the sequence $X_i^\ast$ via $A_\ell^\ast = (X_{2\ell q_n + 1}^\ast,\ldots,X_{(2\ell+1)q_n}^\ast)$ and $B_\ell^\ast = (X_{(2\ell +1)q_n+1}^\ast,\ldots,X_{(2\ell + 2)q_n}^\ast)$ for $\ell = 0,\ldots, p_n-1$, and consider the event $\Xiast = \{ X_i = X_i^\ast \text{ for all } i = 1,\ldots,n \}$.
	In addition consider the event $\Xi$ defined as in the proof of Theorem~\ref{thm:adaptive:xi:unknown}.
	We consider the decomposition
	\begin{align*}
		\E \Vert \lambdahat_\mfhat - \lambda \Vert^2 &= \E \Vert \lambdahat_\mfhat - \lambda \Vert^2 \1_{\Xi \cap \Xiast} + \E \Vert \lambdahat_\mfhat - \lambda \Vert^2 \1_{\Xi^\complement \cap \Xiast} + \E \Vert \lambdahat_\mfhat - \lambda \Vert^2 \1_{\Xiastc}\\
		&\eqdef T_1 + T_2 + T_3,
	\end{align*}
	and bound the three terms separately.
	
	% erster Summand
	\emph{Upper bound for $T_1$}:
	Following along the lines of the proofs of Theorems~\ref{thm:adaptive:xi:known} and~\ref{thm:adaptive:xi:unknown} one can show
	\begin{align*}
		\E \Vert \lambdahat_{\mfhat} - \lambda \Vert^2 \1_{\Xi \cap \Xiast} &\leq C\min_{\mf \in \Mc_n} \{\Vert \lambda_\mf - \lambda \Vert^2 + \pen(\mf)\}\\
		&+16\sum_{\mf' \in \Mc_n} \E \left[ \left(  \sup_{t \in \Bc_{\mf'}} \vert \langle \Thetahat_n,t \rangle \vert^2 - 50\mu \cdot \frac{\Phi^2\D_{\mf'} \log(n+2)}{n} \right)_+ \right]\\
		&+16\sum_{\mf' \in \Mc_n} \E \left[ \left( \sup_{t \in \Bc_{\mf'}} \vert \langle \Thetatilde_n,t \rangle \vert^2 - 8\Phi^2 \D_{\mf'} \mu \cdot \frac{\sum_{k=0}^\infty \beta_k}{n} \right)_+ \1_{\Xi \cap \Xiast} \right]\\
		&+ 16\sum_{\mf' \in \Mc_n} \bigg( 8 \Phi^2 \D_{\mf'} \mu \frac{\sum_{k=0}^\infty \beta_k}{n} - \frac{\D_{\mf'} \log(n+2)}{8n} \bigg)_+. 
	\end{align*}
	Under the given assumptions on the potential models we obtain for the last term the estimate
	\begin{align*}
		\sum_{\mf' \in \Mc_n} \bigg( 8 \Phi^2 \D_{\mf'} \mu \cdot \frac{\sum_{k=0}^\infty \beta_k}{n} - \frac{\D_{\mf'} \log(n+2)}{8n} \bigg)_+ &\leq \sum_{\substack{\mf' \in \Mc_n \\ \D_{\mf'} \leq m_0}} 8 \Phi^2 \D_{\mf'} \mu \cdot \frac{\sum_{k=0}^\infty \beta_k}{n}\\
		&\leq \sum_{m=1}^{m_0} 8\Phi^2 m \mu \cdot \frac{\sum_{k=0}^\infty \beta_k}{n}\\
		&=4\Phi^2 \mu \left( \sum_{k=0}^{\infty} \beta_k \right) \frac{m_0 (m_0+1)}{n}
	\end{align*}
	where $m_0=m_0(\Phi^2, \mu, (\beta_k)_{k \in \N_0})$ is some non-negative integer that depends on $\mu$, $\Phi$, and the sequence of $\beta$-mixing coefficients.
	The term $E_{\mf' 1} = \E [ ( \sup_{t \in \Bc_\mf'} \vert \langle \Thetahat_n,t \rangle \vert^2 - 50\mu \cdot \frac{\Phi^2\D_{\mf'} \log(n+2)}{n} )_+ ]$ can be bounded exactly as in the proof of Theorem~\ref{thm:adaptive:xi:known} via Lemma~\ref{conc:Theta:hat} (since the investigation is based on conditioning on the covariates $X_1,\ldots,X_n$):
		\begin{align*}
			\E \left[ \left( \sup_{t \in \Bc_{\mf'}} \vert \langle \widehat \Theta_n,t \rangle \vert^2 - 50\mu \cdot\frac{\Phi^2 \D_{\mf'} \log(n+2)}{n} \right)_+ \right] &\leq K_1' \left\lbrace \frac{\D_{\mf'}}{n} \exp(-2 \log(n+2))\right.\\
			&\hspace{1em}+\left. \frac{\D_{\mf'}}{n^2} \exp(-K_2'\sqrt{n}) \right\rbrace.
		\end{align*}
	In order to treat the term $E_{\mf'2} = \E [ ( \sup_{t \in \Bc_{\mf'}} \vert \langle \Thetatilde_n,t \rangle \vert^2 - 8 \Phi^2 \D_{\mf'} \mu \cdot \frac{\sum_{k=0}^\infty \beta_k}{n} )_+ \1_{\Xi \cap \Xiast} ]$ we use the decomposition $\langle \Thetatilde_n,t \rangle \1_{\Xi \cap \Xiast} = ( \nutilde^\ast_1(t) +  \nutilde^\ast_2(t))/2 \cdot \1_{\Xi \cap \Xiast}$ where
	\begin{align*}
		\nutilde^\ast_1(t) &= \frac{1}{p_n} \sum_{\ell = 0}^{p_n - 1} \frac{1}{q_n} \sum_{i=2\ell q_n + 1}^{(2\ell+1)q_n} \{ r_t(X_i^\ast) - \E[r_t(X_i^\ast)]\},\\
		\nutilde^\ast_2(t) &= \frac{1}{p_n} \sum_{\ell = 0}^{p_n - 1} \frac{1}{q_n} \sum_{i=(2l+1)q_n + 1}^{(2l+2)q_n} \{ r_t(X_i^\ast) - \E[r_t(X_i^\ast)]\},
	\end{align*}
	and $r_t$ is defined as in the proof of Lemma~\ref{conc:Theta:tilde} (note that it is possible to replace $X_i$ with $X_i^\ast$ by the definition of $\Xiast$).
	By Lemma~\ref{lem:conc:nutilde} we have for $i= 1,2$
\begin{equation*}
	\E \left[ \left( \sup_{t \in \Bc_{\mf'}} \vert \nutilde^\ast_i(t) \vert^2 - \frac{8\Phi^2 \D_{\mf'} \mu \sum_{k=0}^{\infty} \beta_k}{n} \right)_+ \right] \leq c_1 \left\lbrace \frac{\D_{\mf'}}{n} \exp(-c_2q_n) + \frac{\D_{\mf'}}{p_n^2} \exp \left( - c_3 \sqrt{n} \right) \right\rbrace.
\end{equation*}
Hence,\begin{align*}
		E_{\mf' 2} &\leq \frac{1}{2} \E \left[ \left( \sup_{t \in \Bc_{\mf'}} \vert \nutilde^\ast_1(t) \vert^2 - \frac{8\Phi^2 \D_{\mf'} \mu \sum_{k=0}^{\infty} \beta_k}{n} \right)_+ \right]\\
		&\hspace{1em}+ \frac{1}{2} \E \left[ \left( \sup_{t \in \Bc_{\mf'}} \vert \nutilde_2^\ast(t) \vert^2 - \frac{8\Phi^2 \D_{\mf'} \mu \sum_{k=0}^{\infty} \beta_k}{n} \right)_+ \right]\\
		&\lesssim \frac{\D_{\mf'}}{n} \cdot \exp(-c_2q_n) + \frac{\D_{\mf'}}{p_n^2} \exp \left( - c_3 \sqrt{n} \right),
	\end{align*}
and therefore
	\begin{align*}
		\E \Vert \lambdahat_{\mfhat} - \lambda \Vert^2 \1_{\Xi \cap \Xiast} &\lesssim \min_{\mf \in \Mc_n} \{\Vert \lambda_\mf - \lambda \Vert^2 + \pen(\mf)\}\\
		&\hspace{1em}+ \sum_{\mf' \in \Mc_n} \left[ \frac{\D_{\mf'}}{n} \exp(-2 \log(n+2)) + \frac{\D_{\mf'}}{n^2} \exp(-K_2'\sqrt{n})\right] \\
		&\hspace{1em}+ \sum_{\mf' \in \Mc_n} \left[  \frac{\D_{\mf'}}{n} \cdot \exp(-c_2q_n) + \frac{\D_{\mf'}}{p_n^2} \exp \left( - c_3 \sqrt{n} \right) \right] + \frac{1}{n}.
	\end{align*}
	Finally, by exploiting Assumption~\ref{ass:model:adaptive} on the class of models and the definition of $q_n$ we obtain
	\begin{equation*}
		\E \Vert \lambdahat_{\mfhat} - \lambda \Vert^2 \1_{\Xi \cap \Xiast} \lesssim  \min_{\mf \in \Mc_n} \{ \Vert \lambda_\mf - \lambda \Vert^2 + \pen(\mf) \} + \frac{1}{n}.
	\end{equation*}
	
	% zweiter Summand
	\emph{Upper bound for $T_2$}:
	As in the proof of Theorem~\ref{thm:adaptive:xi:unknown} one derives the estimate
	\begin{equation*}
		 \E \Vert \lambdahat_\mfhat - \lambda \Vert^2 \1_{\Xi^\complement \cap \Xiast} \leq 2 \Phi^2c_\mf n T_4(\Vert \lambda \Vert_\infty)^{1/2} \P (\Xi^\complement \cap \Xi^\ast)^{1/2}  + 2 \Vert \lambda \Vert^2 \P (\Xi^\complement \cap \Xi^\ast).
	\end{equation*}
	Grant to this estimate it suffices to show $\P (\Xi^\complement \cap \Xi^\ast) \lesssim n^{-4}$.
	Proceeding as in the proof of Theorem~\ref{thm:adaptive:xi:unknown}, we obtain
	\begin{align*}
		\P (\Xi^\complement \cap \Xi^\ast) &\leq \sum_{j=1}^{M} \P( \{\nu_n(\phi_j) \geq \xi\} \cap \Xiast) + \P(\{-\nu_n(\phi_j) \geq \xi\} \cap \Xiast)\\
		&\hspace{1em}+ \P(\{\nutilde_n(\phi_j) \geq \xi\} \cap \Xiast)  + \P(\{-\nutilde_n(\phi_j) \geq \xi\} \cap \Xiast)
	\end{align*}
	where $\nu_n(\phi_j)$ and $\nutilde_n(\phi_j)$ are defined as in the proof of Theorem~\ref{thm:adaptive:xi:unknown}.
	The probabilities $\P(\{\pm \nu_n(\phi_j) \geq \xi \} \cap \Xi^\ast)$ can be handled with exactly as in the proof of Theorem~\ref{thm:adaptive:xi:unknown} by conditioning on the covariates $X_i$.
	In order to bound the terms $\P(\{ \pm \nutilde_n(\phi_j) \geq \xi \} \cap \Xi^\ast)$, we consider the $\nutilde_n^\ast(\phi_j)$ defined exactly as $\nutilde_n(\phi_j)$ with $X_i$ replaced with $X_i^\ast$.
	Note that by construction $\P(\{\pm \nutilde_n(\phi_j) \geq \xi \} \cap \Xi^\ast) = \P(\{\pm \nutilde_n^\ast(\phi_j) \geq \xi \} \cap \Xi^\ast)$.
	We have the decomposition $\nutilde^\ast_n(\phi_j) = \frac{1}{2} \nutilde^\ast_{1n}(\phi_j) + \frac{1}{2} \nutilde^\ast_{2n}(\phi_j)$ where
	\begin{align*}
		\nutilde^\ast_{1n}(\phi_j) &= \frac{1}{p_n} \sum_{\ell = 0}^{p_n-1} \frac{1}{q_n} \sum_{i=2\ell q_n+1}^{(2\ell+1)q_n} \lambda(X_i^\ast)\phi_j(X_i^\ast) - \int_\X \lambda(x)\phi_j(x) \P(dx), \\
		\nutilde^\ast_{2n}(\phi_j) &= \frac{1}{p_n} \sum_{\ell = 0}^{p_n-1} \frac{1}{q_n} \sum_{i=(2\ell+1)q_n +1}^{(2\ell+2)q_n} \lambda(X_i^\ast)\phi_j(X_i^\ast) - \int_\X \lambda(x)\phi_j(x) \P(dx).
	\end{align*}
	Note that the separate summands of the outer sum are independent.
	Hence, applying Bernstein's inequality (see Lemma~\ref{PR:prop:bernstein}) yields for $i = 1,2$ that
	\begin{equation*}
		\P (\{\pm \nutilde^\ast_{in}(\phi_j) \geq \xi\} \cap \Xiast) \leq \exp \left( - \frac{c_\Pi p_n}{64M} \right) 
	\end{equation*}
	which implies under the stated assumptions that
	\begin{equation*}
		\P (\Xi^\complement \cap \Xi^\ast) \leq 4M \exp \left( - \frac{c_\Pi p_n}{64M} \right) \lesssim n^{-4}.
	\end{equation*}
	
	% dritter Summand
	\emph{Upper bound for $T_3$}:
	Following along the lines of the proof of the upper bound for term $T_2$ in the proof of Theorem~\ref{thm:adaptive:xi:unknown} one can show
	\begin{equation*}
		\E  \Vert \lambdahat_\mfhat - \lambda \Vert^2 \1_{\Xi^{\ast\complement}} \leq 2\Phi^2c_\mf n T_4(\Vert \lambda \Vert_\infty) ^{1/2}\P(\Xi^{\ast\complement})^{1/2} + 2 \Vert \lambda \Vert^2 \P (\Xi^{\ast\complement})
	\end{equation*}
	with $T_4$ being the fourth Touchard polynomial,
	and it suffices again to show that $\P (\Xi^{\ast\complement}) \lesssim n^{-4}$.
	It holds
	\begin{equation*}
		\P(\Xi^{\ast\complement}) \leq 2p_n\beta_{q_n} = 2n q_n^{-1} \beta_{q_n} \lesssim n^{-4}
	\end{equation*}
	where the last estimate holds due to Assumption~\ref{ass:ex:lacour} both in the geometric and the arithmetic case.

\section{Technical lemmata}

\begin{lemma}\label{conc:Theta:tilde}
	For all $\mf \in \Mc_n$, we have
	\begin{equation*}
	\E \left[  \left( \sup_{t \in \Bc_{\mf}} \vert \langle \widetilde \Theta_n,t \rangle \vert^2 - 3 \mu \cdot \frac{\Phi^2 \D_\mf}{n} \right)_+ \right] \leq K_1 \left\lbrace \frac{1}{n} \exp(-K_2 \D_\mf) + \frac{\D_\mf}{n^2} \exp(-K_3 \sqrt n) \right\rbrace 
	\end{equation*}
	with strictly positive numerical constants $K_1$, $K_2$, and $K_3$.
	\end{lemma}
	
	\begin{proof}
		With $t \in \Bc_\mf$, we associate the function
		\begin{equation*}
		r_t(x) \defeq \sum_{\eta \in \Ic_\mf} \tau_\eta \lambda(x) \phi_\eta(x)
		\end{equation*}
		where the $\tau_\eta = \int_{\X} \phi_\eta(x) t(x)\P(dx)$ denote for $\eta \in \Ic_\mf$ the coefficients of the function $t$ in terms of the basis given by the $\phi_\eta$.
		Evidently, we have $\E [r_t(X)] = \sum_{\eta \in \Ic_\mf} \tau_\eta \theta_\eta$.
		Consequently, one has the identity
		\begin{equation*}
		\langle \widetilde \Theta_n, t \rangle = \frac{1}{n} \sum_{i=1}^{n} r_t(X_i) - \E [r_t(X_i)],
		\end{equation*}
		and $\langle \widetilde \Theta_n, t \rangle$ will take the role of $\nu_n$ in Lemma~\ref{PR:lem:ex:talagrand}.
		We now check the preconditions concerning the existence of suitable constants $M_1$, $H$ and $\upsilon$ in the statement of Lemma~\ref{PR:lem:ex:talagrand}.
		
		\noindent \emph{Condition concerning $M_1$:} We have
		\begin{align*}
		\sup_{t \in \Bc_\mf} \Vert r_t \Vert_\infty^2 = \sup_{t \in \Bc_\mf} \sup_{y \in [0,1]} \vert r_t(y) \vert^2 &\leq \sup_{t \in \Bc_\mf} \sup_{y \in [0,1]} \left( \sum_{\eta \in \Ic_\mf} \tau_\eta^2 \right) \left( \sum_{\eta \in \Ic_\mf} \lambda^2(y) \cdot \phi^2_\eta(y) \right) \\
		&\leq \Vert \lambda \Vert_\infty^2 \Phi^2 \D_\mf \leq \mu  \Phi^2 \D_\mf,
		\end{align*}
		and we can put $M_1 \defeq ( \mu \Phi^2 \D_\mf )^{1/2}$.
		
		\noindent \emph{Condition concerning $H$:} We have 
		\begin{align*}
		\E [ \sup_{t \in \Bc_\mf} \vert \langle \widetilde \Theta_n,t \rangle \vert^2] &\leq \frac{1}{n^2} \E \left[ \sup_{t \in \Bc_\mf} \left( \sum_{\eta \in \Ic_\mf} \tau_\eta^2 \right) \left( \sum_{\eta \in \Ic_\mf} \left| \sum_{i=1}^{n} \left\lbrace \phi_\eta(X_i) \lambda(X_i) - \theta_\eta \right\rbrace \right|^2 \right) \right] \\
		&= \frac{1}{n} \sum_{\eta \in \Ic_\mf} \Var \left( \phi_\eta(X_1) \lambda(X_1) \right) \leq \frac{1}{n} \sum_{\eta \in \Ic_\mf} \E [ \left( \phi_\eta(X_1) \lambda(X_1) \right)^2 ] \\
		&\leq \frac{\Phi^2 \D_\mf}{n} \cdot \Vert \lambda \Vert^2_\infty \leq \frac{\mu  \Phi^2 \D_\mf}{n},
		\end{align*}
		and thus by Jensen's inequality we can put $H \defeq \left( \frac{\mu \Phi^2 \D_\mf}{n} \right)^{1/2}$.
		
		\noindent \emph{Condition concerning $\upsilon$:} For arbitrary $t \in \Bc_\mf$, it holds
		\begin{align*} 
		\Var \left( r_t(X) \right) = \Var \left( \sum_{\eta \in \Ic_\mf} \tau_\eta \lambda(X) \phi_\eta(X) \right)
		&\leq \E \left[ \left( \sum_{\eta \in \Ic_\mf} \tau_\eta \lambda(X)\phi_\eta(X) \right)^2 \right] \leq \mu.
		\end{align*}
		Thus, we can take $\upsilon \defeq \mu$ and the statement of the lemma follows now by applying Lemma~\ref{PR:lem:ex:talagrand} with $\epsilon = \frac{1}{4}$.
		\end{proof}
			
\begin{lemma}\label{conc:Theta:hat}
With the notation from the proof of Theorem~\ref{thm:adaptive:xi:known} it holds for all $\mf \in \Mc_n$
\begin{align*}
\E \left[ \left( \sup_{t \in \Bc_{\mf}} \vert \langle \widehat \Theta_n,t \rangle \vert^2 - 50\mu \cdot\frac{\Phi^2 \D_\mf \log(n+2)}{n} \right)_+ \right] &\leq K_1' \left\lbrace \frac{\D_\mf}{n} \exp(-2 \log(n+2))\right.\\
&\hspace{1em}+\left. \frac{\D_\mf}{n^2} \exp(-K_2'\sqrt{n}) \right\rbrace
\end{align*}
with strictly positive numerical constants $K_1'$ and $K_2'$.
\end{lemma}
				
\begin{proof}
Given $\boldsymbol X = (X_1,\ldots,X_n)$, we can write $Y_i$ as $\int_{0}^{1} dN_i(s)$ where $N_i$ is a Poisson process with homogeneous intensity equal to $\lambda(X_i)$.
Thus, conditional on $\boldsymbol X$, it holds
\begin{align*}
\langle \widehat \Theta_n , t \rangle &= \frac{1}{n} \sum_{\eta \in \Ic_\mf} \tau_\eta \sum_{i=1}^{n}\left\lbrace \int_0^1 \phi_\eta(X_i) dN_i(s) - \phi_\eta(X_i) \cdot \lambda(X_i) \right\rbrace\\
&= \frac{1}{n} \sum_{i=1}^{n} \left\lbrace \int_0^1 r_t(s)dN_i(s) - \int_0^1 r_t(s) \lambda(X_i)ds \right\rbrace 
\end{align*}
where $r_t$ is the function given by $r_t(s) \defeq \sum_{\eta \in \Ic_\mf} \tau_\eta \phi_\eta(X_i)$ (note that, given $\boldsymbol X$, this is a constant function).
We now check the preconditions concerning the existence of suitable constants $M_1$, $H$ and $\upsilon$ from Lemma~\ref{lem:conc:cox}.
					
\noindent \emph{Condition concerning $M_1$}: We have
\begin{align*}
\sup_{t \in \Bc_\mf} \Vert r_t \Vert_\infty^2 = \sup_{t \in \Bc_\mf} \left(  \sum_{\eta \in \Ic_\mf} \tau_\eta \phi_\eta(X_i) \right) ^2 &\leq \sup_{t \in \Bc_\mf} \left(  \sum_{\eta \in \Ic_\mf} \tau_\eta^2 \right)  \cdot \left(  \sum_{\eta \in \Ic_\mf} \phi^2_\eta(X_i) \right)  \leq \Phi^2 \D_\mf,
\end{align*}
and we can take $M_1 = \sqrt{\mu \cdot \Phi^2 \D_\mf}$.
					
\noindent \emph{Condition concerning $H$}: It holds
\begin{align*}
\E [ \sup_{t \in \Bc_\mf} \vert \langle \widehat \Theta_n , t \rangle \vert^2 | \boldsymbol X ]&\\
&\hspace{-4em}\leq \sup_{t \in \Bc_\mf} \left( \sum_{\eta \in \Ic_\mf} \tau_\eta^2 \right)  \E \left[  \sum_{\eta \in \Ic_\mf} \Big\vert \frac{1}{n} \sum_{i=1}^n \left\lbrace  \int_{0}^{1} \phi_\eta(X_i)[dN_i(s)- \lambda(X_i)ds] \right\rbrace \Big\vert^2 \Bigg \vert \boldsymbol X \right] \\
&\hspace{-4em} \leq \frac{1}{n} \sum_{\eta \in \Ic_\mf} \Var \left( \int_{0}^{1} \phi_\eta(X_1)dN_1(s) \Big\vert X_1 \right)\\
&\hspace{-4em}=\frac{1}{n}\sum_{\eta \in \Ic_\mf} \int_0^1 \phi_\eta^2(X_1) \lambda(X_1) ds\\
&\hspace{-4em}\leq \frac{\Phi^2 \D_\mf}{n} \cdot \Vert \lambda \Vert_\infty \leq \frac{\Phi^2 \D_\mf}{n} \cdot \mu.   
\end{align*}
Thus, we can put $H \defeq \left( \frac{\Phi^2 \D_\mf \mu \log(n+2)}{n}  \right)^{1/2}$ (the additional enlargement of $H$ by the logarithmic term is needed in the proof of Theorem~\ref{thm:adaptive:xi:known}).
					
\noindent \emph{Condition concerning $\upsilon$}: For arbitrary $\mf \in \Mc_n$ and $t \in \Bc_\mf$ it holds
\begin{align*}
\Var \left( \int_0^1 r_t(s)dN_k(s) | X_k\right) &= \int_0^1 \vert r_t(s) \vert^2 \lambda(X_k) ds \leq \Vert \lambda \Vert_\infty \cdot \Vert r_t \Vert_\infty^2 \leq \Vert \lambda \Vert_\infty \Phi^2 \D_\mf,
\end{align*}
and we can put $\upsilon \defeq \mu \Phi^2 \D_\mf$.
					
We can apply Lemma~\ref{lem:conc:cox} with $\epsilon=12$ which yields
\begin{align*}
	\E \left[ \left( \sup_{t \in \Bc_{\mf}} \vert \langle \widehat \Theta_n,t \rangle \vert^2 - 50 \mu\cdot \frac{\Phi^2  \D_\mf \log(n+2)}{n} \right)_+ \Bigg \vert \boldsymbol X \right] &\leq \\ &\hspace{-18em} K_1' \left\lbrace \frac{\Phi^2\D_\mf \mu}{n} \exp(-2 \log(n+2)) + \frac{\Phi^2 \D_\mf \mu}{n^2} \exp(-K_2' \sqrt{n \log(n+2)}) \right\rbrace.
\end{align*}
Since the right-hand side of the last estimate does not depend on $\boldsymbol X$, taking expectations on both sides implies the assertion of the lemma.
\end{proof}
					
\begin{lemma}\label{lem:conc:nutilde}
	For $\nutilde^\ast_i(\cdot)$ defined as in the proof of Theorem~\ref{thm:adaptive:dep} it holds for $i=1,2$ that
	\begin{equation*}
	\E \left[ \sup_{t \in \Bc_\mf} \left( \vert \nutilde^\ast_i(t) \vert^2 - \frac{8\Phi^2 \D_\mf \mu (\sum_{k=0}^{\infty} \beta_k)}{n} \right)_+ \right] \leq c_1 \left\lbrace \frac{\D_\mf}{n}  \exp(-c_2q_n) + \frac{\D_\mf}{p_n^2} \exp \left( - c_3 \sqrt{n} \right) \right\rbrace
	\end{equation*}
	with strictly positive numerical constants $c_1$, $c_2$, and $c_3$.
\end{lemma}

\begin{proof}
	We state the proof for $i=1$ only.
	We want to apply Lemma~\ref{PR:lem:ex:talagrand} to $\nutilde_1 = \frac{1}{p_n} \sum_{\ell=0}^{p_n-1} Z_\ell$ where $Z_\ell = \frac{1}{q_n} \sum_{i=2lq_n+1}^{(2\ell+1)q_n} \{ r_t(X_i^\ast) - \E r_t(X_i^\ast) \}$,
	and the function $r_t$ is defined as in Lemma~\ref{conc:Theta:tilde}.
	Note that $Z_\ell$ and $Z_{\ell'}$ are independent for $\ell \neq \ell'$ by construction.
	Thus, it remains to find constants $M_1$, $H$ and $\upsilon$ satisfying the preconditions of Lemma~\ref{PR:lem:ex:talagrand}.
	
	\emph{Condition concerning $M_1$}: For each $\ell=0,\ldots,p_n-1$ it holds
	\begin{equation*}
		\sup_{t \in \Bc_\mf} \left\lVert \frac{1}{q_n} \sum_{i=2\ell q_n+1}^{(2\ell+1)q_n} r_t \right\rVert_\infty \leq \sup_{t \in \Bc_\mf} \Vert r_t \Vert_\infty \leq \sqrt{\mu \Phi^2 \D_\mf}
	\end{equation*}
	where the last step has already been shown in the proof of Lemma~\ref{conc:Theta:tilde}.
	
	\emph{Condition concerning $H$}: By Lemma~\ref{lem:var:bound:beta} we have
	\begin{align*}
		\E &\left[ \sup_{t \in \Bc_\mf} \left\lvert \frac{1}{p_nq_n} \sum_{\ell = 0}^{p_n-1} \sum_{i=2\ell q_n + 1}^{(2\ell + 1)q_n}\{ r_t(X_i^\ast) - \E r_t(X_i^\ast) \} \right\rvert^2 \right]\\
		&\hspace{14em}\leq \frac{1}{p_n^2q_n^2} \E \left[ \sum_{\eta \in \Ic_\mf} \left\lvert \sum_{\ell=0}^{p_n-1} \sum_{i=2\ell q_n + 1}^{(2\ell + 1)q_n} \{ \phi_\eta(X_i^\ast)\lambda(X_i^\ast) - \theta_\eta \} \right\rvert^2  \right]\\
		&\hspace{14em}= \frac{1}{p_nq_n^2} \sum_{\eta \in \Ic_\mf} \Var \left( \sum_{i=2\ell q_n + 1}^{(2\ell + 1)q_n} \phi_\eta(X_i^\ast) \lambda(X_i^\ast) \right)\\
		&\hspace{14em}\leq \frac{4 \Phi^2  \D_\mf}{p_nq_n} \cdot \left( \sum_{k=0}^{n} \beta_k \right)  \cdot \Vert \lambda \Vert_\infty^2\\
		&\hspace{14em}\leq\frac{8\Phi^2 \D_\mf \Vert \lambda \Vert_\infty^2}{n} \left( \sum_{k=0}^{\infty} \beta_k\right)  \eqdef H^2.
	\end{align*}
	
	\emph{Condition concerning $\upsilon$}: We have
	
	\begin{align*}
		\Var \left( \frac{1}{q_n} \sum_{i=2\ell q_n + 1}^{(2\ell + 1)q_n} r_t(X_i^\ast) \right) &= \frac{1}{q_n^2} \Var \left( \sum_{i=2\ell q_n + 1}^{(2\ell + 1)q_n} \sum_{\eta \in \Ic_\mf} \tau_\eta \lambda(X_i^\ast) \phi_\eta(X_i^\ast) \right)\\
		&\leq \frac{4}{q_n} \E \left[ \left( \sum_{k=0}^{n} b_k \right)  \left( \sum_{\eta \in \Ic_\mf} \lambda^2(X_i^*) \phi_\eta^2(X_i^*) \right)  \right]\\
		&\leq \frac{4\Phi^2}{q_n} \Vert \lambda \Vert_\infty^2 \D_\mf \left( \sum_{k=0}^\infty \beta_k\right)  \eqdef \upsilon.
	\end{align*}
	
	Now, application of Lemma~\ref{PR:lem:ex:talagrand} yields
	\begin{equation*}
		\E \left[ \left( \sup_{t \in \Bc_\mf} \vert \nutilde_i(t) \vert^2 - \frac{8\Phi^2 \D_\mf \mu (\sum_{k=0}^{\infty} \beta_k)}{n} \right)_+ \right] \leq c_1 \left\lbrace \frac{\D_m}{n} \cdot \exp(-c_2q_n) + \frac{\D_\mf}{p_n^2} \exp \left( - c_3 \sqrt{n} \right) \right\rbrace. 
	\end{equation*}
\end{proof}

\section{Bernstein inequality}

The following version of Bernstein's inequality is taken from~\cite{boucheron2016concentration}.

\begin{lemma}[Bernstein's inequality, \cite{boucheron2016concentration}, Corollary~2.11]\label{PR:prop:bernstein}
	Let $X_1,\ldots,X_n$ be independent real-valued random variables with $\vert X_i\vert \leq b$ for some $b>0$ almost surely for all $i \leq n$. Let $S = \sum_{i=1}^{n} (X_i - \E X_i)$ and $\upsilon = \sum_{i=1}^{n} \E [X_i^2]$. Then
	\begin{equation*}
		\P (S \geq t) \leq \exp \left( - \frac{t^2}{2(\upsilon +bt/3)} \right).
	\end{equation*}
\end{lemma}

\section{Concentration inequalities}

\subsection{A useful consequence of Talagrand's inequality}

The following lemma is a consequence from Talagrand's inequality and is taken from~\cite{chagny2015optimal}.
For a detailed proof, we refer to~\cite{chagny2013estimation}.

\begin{lemma}\label{PR:lem:ex:talagrand}
	Let $X_1,\ldots,X_n$ be i.i.d.~random variables with values in some Polish space and define $\nu_n(s) = \frac{1}{n} \sum_{i=1}^{n} s(X_i) - \E [s(X_i)]$, for $s$ belonging to a countable class $\Sc$ of measurable real-valued functions. Then, for any $\epsilon > 0$, there exist positive constants $c_1$, $c_2=\frac{1}{6}$, and $c_3$ such that
	\begin{align*}
		\E \left[ \left( \sup_{s \in \Sc} \vert \nu_n(s) \vert^2 - c(\epsilon)H^2 \right)_+ \right] &\\
		&\hspace{-3em}\leq c_1 \left\lbrace \frac{\upsilon}{n} \exp \left( -c_2 \epsilon \frac{nH^2}{\upsilon}\right) +\frac{M_1^2}{C^2(\epsilon) n^2} \exp\left( -c_3 C(\epsilon) \sqrt{\epsilon} \frac{nH}{M_1} \right) \right\rbrace,
	\end{align*}
	with $C(\epsilon) = (\sqrt{1+\epsilon} -1) \wedge 1$, $c(\epsilon) = 2 (1 + 2 \epsilon)$ and
	\begin{equation*}
		\sup_{s \in \Sc} \Vert s \Vert_\infty \leq M_1, \quad \E [ \sup_{s \in \Sc} \vert \nu_n(s) \vert ] \leq H, \quad \text{and} \quad \sup_{s \in \Sc} \Var(s(X_1)) \leq \upsilon.
	\end{equation*}
\end{lemma}

\subsection{Concentration inequalities for point processes}

The following lemma is taken from~\cite{kroll2017concentration}.

\begin{lemma}\label{lem:conc:cox}
	Let $N_1,\ldots,N_n$ be independent Cox processes driven by \emph{finite} random measures $\eta_1,\ldots,\eta_n$ (that is, given $\eta_i$, $N_i$ is a Poisson point process with intensity measure $\eta_i$) that are conditionally independent given $\eta_1,\ldots,\eta_n$.
	Set $\nu_n(r) = \frac{1}{n} \sum_{k=1}^{n} \{ \int_{\X} r(x)dN_k(x) - \int_{\X} r(x)d\eta_k(x) \}$ for $r$ contained in a countable class of real-valued measurable functions.
	Then, for any $\epsilon > 0$, there exist constants $c_1$, $c_2=\frac{1}{6}$, and $c_3$ such that
	\begin{align*}
		\E \left[ \left( \sup_{r \in \Rc} \vert \nu_n(r) \vert^2 - c(\epsilon)H^2 \right)_+ \Bigg \vert \boldsymbol \eta \right]&\\
		&\hspace{-3em}\leq c_1 \left\lbrace \frac{\upsilon}{n} \exp \left( -c_2 \epsilon \frac{nH^2}{\upsilon} \right) + \frac{M_1^2}{C^2(\epsilon)n^2} \exp \left( - c_3 C(\epsilon) \sqrt{\epsilon} \frac{nH}{M_1} \right)   \right\rbrace 
	\end{align*}
	where $\boldsymbol \eta = (\eta_1,\ldots,\eta_n)$, $C(\epsilon) = (\sqrt{1+\epsilon} - 1) \wedge 1$, $c(\epsilon) = 2(1+2\epsilon)$ and $M_1$, $H$ and $\upsilon$ are such that
	\begin{equation*}
		\sup_{r \in \Rc} \Vert r \Vert_\infty \leq M_1, \quad \E [\sup_{r \in \Rc} \vert \nu_n(r) \vert  | \boldsymbol \eta] \leq H, \quad \sup_{r \in \Rc} \Var \left( \int_{\X} r(x)dN_k(x) \Big \vert \boldsymbol \eta \right) \leq \upsilon \quad \forall k. 
	\end{equation*}
\end{lemma}

The following lemma is a Bernstein type inequality for point processes and taken from~\cite{reynaud2003adaptive}.

\begin{lemma}[\cite{reynaud2003adaptive}, Proposition~7]\label{PR:prop:bernstein:PPP}
	Let $N$ be a Poisson point process on some measurable space $(\X, \Xs)$ with finite intensity measure $\mu$. 
	Let $g$ be a measurable function on $(\X, \Xs)$, essentially bounded, such that $\int_{\X} g^2(x) \mu(dx) > 0$.
	Then
	\begin{equation*}
		\P \left( \int_{\X} g(x)(dN(x) - \mu(dx)) \geq t \right)  \leq \exp \left( - \frac{t^2}{2 (\int_{\X} g^2(x)\mu(dx) + \Vert g \Vert_\infty t/3)} \right), \qquad t > 0.
	\end{equation*}
\end{lemma}
% Formatierung der Ausgabe von biblatex anpassen!
\printbibliography

\end{document}